\documentclass[a4paper,11pt]{article}
\usepackage[T1]{fontenc}
\usepackage{amsfonts}
\usepackage{amsmath}
\usepackage{enumerate}
\usepackage{amssymb}
\usepackage{amscd}
\usepackage{amsthm}
\usepackage{fontenc}
\usepackage{cancel}
\usepackage[pdftex]{graphicx}
\usepackage[top=3cm,bottom=3cm,left=3cm,right=3cm]{geometry}

\newtheorem{theo}{Theorem}
\newtheorem{defi}{Definition}
\newtheorem{prop}{Proposition}
\newtheorem{cor}{Corollary}
\newtheorem{lem}{Lemma}
\theoremstyle{remark}
\newtheorem*{rem}{Remark}

\newcommand{\Om}{\Omega}
\newcommand{\R}{\mathbb R}
\newcommand{\C}{C}
\newcommand{\Rn}{{\mathbb R^d}}
\newcommand{\tr}{\operatorname{Tr}}
\renewcommand{\div}{\operatorname{div}}
\newcommand{\nau}{\nabla u}
\newcommand{\vnau}{\vert \nabla u \vert}
\newcommand{\naue}{\nabla u_\varepsilon}
\newcommand{\vnaue}{\vert \nabla u_\varepsilon \vert}
\newcommand{\ls}{\leqslant}
\newcommand{\scal}[2]{\left( #1\, ,\, #2 \right)}
\newcommand{\scalv}[2]{\left \langle #1\, ,\, #2 \right \rangle}
\newcommand{\Deltas}{\Delta^S}
\newcommand{\nablas}{\nabla^S}
\newcommand{\x}{\mathbf x}

\newcommand{\gs}{\geqslant}
\newcommand{\gep}{k_\varepsilon}
\newcommand{\eps}{\varepsilon}

\date{}

\title{Mean curvature flow with obstacles: \\
existence, uniqueness and regularity of solutions}

\author{
Gwenael Mercier\footnote{CMAP, École polytechnique, Palaiseau, France, 
email: gwenael.mercier@cmap.polytechnique.fr},
Matteo Novaga\footnote{Dipartimento di Matematica, 
Universit\`a di Pisa, Largo Bruno Pontecorvo 5, 56127 Pisa, Italy; e-mail: novaga@dm.unipi.it}}

\begin{document}

\maketitle

\begin{abstract}
We show short time existence and uniqueness of $\C^{1,1}$ solutions 
to the mean curvature flow with obstacles, when the obstacles are 
of class $\C^{1,1}$. If the initial interface is a periodic graph 
we show long time existence of the evolution and convergence 
to a minimal constrained hypersurface.
\end{abstract}

\tableofcontents

\section{Introduction and main results}
Mean curvature flow is a prototypical geometric evolution, arising in many models
from Physics, Biology and Material Science, as well as in a variety of mathematical problems. 
For such a reason, this flow has been widely studied in the past years, 
starting from the pioneristic work of K. Brakke \cite{brakke78}
(we refer to \cite{gage86,huisken84,ecker89,evans91,chen91} for a far from complete list of references).

In some models, one needs to include the presence of hard obstacles, which the evolving surface 
cannot penetrate (see for instance \cite{ESV} and references therein). 
This leads to a double obstacle problem for the mean curvature flow, which reads 
\begin{equation}\label{eqmo}
v\,=\,H \qquad \text{on }M_t\cap U,
\end{equation}
with constraint
\begin{equation}\label{eqcon}
M_t\subset \overline U \qquad \text{for all }t,
\end{equation}
where $v,\,H$ denote respectively the normal velocity and $d$ times the mean curvature of the interface $M_t$, and the open set $U\subset\R^{d+1}$ represents the obstacle.
Notice that, due to the presence of obstacles, 
the evolving interface is in general only of class $C^{1,1}$ 
in the space variable,
differently from the unconstrained case where it is analytic (see 
\cite{ilm98}). 
While the regularity of parabolic obstacle problems is relatively well understood 
(see \cite{sha08} and references therein), a satisfactory existence and uniqueness 
theory for solutions is still missing.

In \cite{chambolle12} (see also \cite{spadaro11}) the authors approximate such an obstacle problem
with an implicit variational scheme introduced in \cite{ATW,LS}. As a byproduct, 
they prove global existence of weak (variational) solutions, and 
short time existence and uniqueness of regular solutions in the two-dimensional case.
In \cite{mercier} the first author adapts to this setting
the theory of viscosity solutions introduced in \cite{c92user,chen91}, 
and constructs globally defined continuous (viscosity) solutions.  

Let us now state the main results of this paper.

\begin{theo}\label{main}
Let $M_0\subset U$ be an initial hypersurface, and assume that 
both $M_0$ and $\partial U$ are uniformly of class $C^{1,1}$,
with ${\rm dist}(M_0,\partial U)>0$.
Then there exists $T>0$ and a unique solution $M_t$ to \eqref{eqmo},
\eqref{eqcon} on $[0,T)$, such that $M_t$ is of class $C^{1,1}$
for all $t\in [0,T)$.
\end{theo}

Notice that Theorem \ref{main} extends a result in \cite{chambolle12} to 
dimensions greater than two.

When the hypersurface $M_t$ can be written as the graph 
of a function $u(\cdot,t):\R^d\to \R$, equation \eqref{eqmo} reads
\begin{equation}\label{eqg}
u_t = \sqrt{1+\vnau^2}\,\div \left( \frac{\nau}{\sqrt{1+\vnau^2}} \right). 
%\qquad \text{on }\{(x,t):\ \psi(x)<u(x,t)<\psi^+(x)\}.
\end{equation}
If the obstacles are also graphs, the constraint \eqref{eqcon} can be 
written as 
\begin{equation}\label{eqcongr}
\psi^-\ls u \ls \psi^+\,, 
\end{equation}
where the functions $\psi^\pm:\Rn\to \R$ denote the obstacles.

%The second result of this papers is the following
\begin{theo}\label{mainbis}
Assume that $\psi^\pm\in C^{1,1}(\R^d)$, and
let $u_0\in C^{1,1}(\R^d)$ satisfy \eqref{eqcongr}.
Then there exists a unique (viscosity) solution $u$ of 
\eqref{eqg}, \eqref{eqcongr}
on $\R^d\times [0,+\infty)$, such that 
\begin{eqnarray*}
\|\nabla u(\cdot,t)\|_{L^\infty(\R^d)}&\le& \max\left(
\|\nabla u_0\|_{L^\infty(\R^d)} , \|\nabla \psi^\pm\|_{L^\infty(\R^d)}
\right)
\\
\\
\|u_t(\cdot,t)\|_{L^\infty(\R^d)}&\le&
\left\|\sqrt{1+|\nau_0|^2}\div \left( \frac{\nau_0}{\sqrt{1+|\nau_0|^2}} \right)
\right\|_{L^\infty(\R^d)}
\end{eqnarray*}
for all $t>0$.
Moreover $u$ is also of class $C^{1,1}$ uniformly on $[0,+\infty)$. 
\end{theo}

We observe that Theorem \ref{mainbis} extends previous results by Ecker and Huisken \cite{ecker89} in the unconstrained case (see also \cite{cesnov13}). 

\begin{theo}\label{maintris}
Assume that $u_0$ and $\psi^\pm$ are $Q$-periodic,
with periodicity cell $Q=[0,L]^d$, for some $L>0$.
Then the solution $u(\cdot,t)$ of \eqref{eqg}, \eqref{eqcongr} 
is also $Q$-periodic. Moreover
there exists a sequence $t_n \to +\infty$ such that $u(\cdot,t_n)$ converges uniformly as $n\to +\infty$ to a stationary solution to \eqref{eqg}, \eqref{eqcongr}.
\end{theo}

Our strategy will be to approximate the obstacles with ``soft obstacles'' modeled by 
a sequence of uniformly bounded forcing terms. Differently from \cite{chambolle12}, where the existence of 
regular solution is derived from variational estimates on the approximating scheme,
we obtain estimates on the evolving interface,
in the spirit of \cite{ecker91,ecker91b,cesnovval11}, which are uniform in the forcing terms.

\subsection*{Acknowledgements}
We wish to thank to Antonin Chambolle 
for interesting discussions and useful comments on this work. 
The first author was partially supported by the ANR-12-BS01-0014-01 Project Geometrya.

\section{Mean curvature flow with a forcing term}

\subsection{Evolution of geometric quantities}

Let $M$ be a complete orientable $d$-dimensional Riemannian manifold without boundary,
let $F(\cdot,t) : M \to \R^{d+1}$ be a smooth family of immersions,
and denote by $M_t$ the image $F(M,t).$ 
Since $M_t$ is orientable, we can write $M_t=\partial E(t)$ where $E(t)$
 is a family of open subsets of $\R^{d+1}$ depending smoothly on $t$.
We say that $M_t$ evolves by mean curvature with forcing term $k$ if
\begin{equation}\label{evolg}
\frac{d}{dt} F(p,t) = - \big(H(p,t)+ k(F(p,t))\big)\, \nu(p,t),
\end{equation}
where $k:\R^{d+1}\to \R$ is a smooth forcing term, $\nu$ is the unit normal 
to $M_t$ pointing outside $E(t)$, and $H$ is ($d$ times) 
the mean curvature of $M_t$,
with the convention that $H$ is positive whenever $E(t)$ is convex.

We shall compute the evolution of some relevant geometric quantities
under the law \eqref{evolg}. 
We denote by $\nablas,\,\Deltas$ respectively the covariant derivative and 
the Laplace-Beltrami operator on $M$.
As in \cite{huisken84}, 
the metric on $M_t$ is denoted by $g_{ij}(t)$, it inverse is $g^{ij}(t)$,
the scalar product (or any tensors contraction using the metric) on $M_t$ is denoted by $\scalv{\cdot}{\cdot}$ whereas the ambiant scalar product is $\scal{\cdot}{\cdot}$,
the volume element is $\mu_t$, and
the second fondamental form is $A$. 
In particular we have
$A\left(\partial_i,\partial_j\right)=h_{ij}$, 
where we set for simplicity $\partial_i=\frac{\partial}{\partial x_i}$, 
and $H=h_{ii}$, using the Einstein notations (we implicitly sum every index which appears twice in an expression).
We also denote by $\lambda_1,\ldots,\lambda_d$ the eigenvalues of $A$.

Notice that, in terms of the parametrization $F$, we have
\begin{equation}\label{eqij}
g_{ij} = \scal{\partial_i F}{\partial_j F},
\qquad
h_{ij} = -\scal{\partial^2_{ij} F}{\nu}
\qquad \text{for all }i,j\in \{1,\ldots,d\}.
\end{equation}

\begin{prop}
The following equalities hold:
\begin{eqnarray}\label{metricevol}
\frac{d}{dt} g_{ij} &=& -2(H+k) h_{ij}
\\\label{normalevol}
\frac{d}{dt} \nu &=& \nablas (H+k)
\\\label{eqmeas}
\frac{d}{dt} \mu_t &=& -H(H+k) \mu_t
\\\label{secondevol}
\frac{d}{dt} h_{ij} &=& \Deltas h_{ij} +\nablas_i \nablas_j k-2H h_{il}g^{lm}h_{mj} 
-k g^{ml} h_{im} h_{jl} +|A|^2 h_{ij}
\\\label{mcevol}
 \frac{d}{dt} H &=& \Deltas (H+k) +(H+k) |A|^2 
\\\label{evoln2FF}
\frac{d}{dt} |A|^2 &=& \Deltas |A|^2 +2 k g^{ij}g^{sl}g^{mn} h_{is} h_{lm} h_{nj} 
+ 2|A|^4 - 2 |\nablas A|^2+ 2 \scalv{A}{(\nablas)^2 k}.
\end{eqnarray}
\end{prop}

\begin{proof}
The proof follows by direct computations as in 
\cite{huisken84,ecker91b}. 
Recalling \eqref{eqij}, we get
\begin{eqnarray*}
\frac{d}{dt} g_{ij} &=&
\frac{d}{dt}\scal{\partial_i F}{\partial_j F} = 
-(H+k)\left(
\scal{\partial_i \nu}{\partial_j F}
+\scal{\partial_i F}{\partial_j \nu}
\right) 
= -2(H+k) h_{ij}
\\
\frac{d}{dt} \nu &=& \scal{\frac{d}{dt} \nu}{\partial_i F} 
g^{ij}{\partial_j F}  = - \scal{\nu}{\frac{d}{dt} \partial_i F} 
g^{ij}{\partial_j F} \\\nonumber
&=& \scal{\nu}{\partial_i ((H+k)\nu)}g^{ij}\partial_j F 
= \partial_i (H+k) g^{ij} {\partial_j F} = \nablas (H+k).
\end{eqnarray*}

The evolution of the measure on $M_t$ $$\mu_t = \sqrt{\det [g]}$$ is given by
\begin{eqnarray*}
 \frac{d}{dt} \sqrt{\det [g]} & = & \frac{\frac{d}{dt} \det [g]}{2 \sqrt{\det [g]}} = \frac{\det [g] \cdot \tr \left(g^{ij}\,\frac{d}{dt} g_{ij}\right)}{2 \sqrt{\det [g]}}
\\
&=&- \sqrt{\det [g]} \cdot (H+k) g^{ij} h_{ji} = -\mu_t H(H+k).
\end{eqnarray*}

In order to prove \eqref{secondevol} we compute (as usual, we denote the Christoffel symbols by $\Gamma_{ij}^k$)
\begin{eqnarray}\nonumber
 \frac{d}{dt} h_{ij} &=& 
- \frac{d}{dt} \scal{\nu}{\partial^2_{ij} F} 
\\\nonumber
&=& -\scal{\nablas (H+k)}{\partial^2_{ij} F} 
+ \scal{\partial^2_{ij} (H+k)\nu}{\nu} \\\nonumber
&=&- \scal{g^{kl} \partial_k(H+k) \partial_l F}{\Gamma_{ij}^k \partial_k F - h_{ij} \nu} 
\\\nonumber
&& + \partial^2_{ij} (H+k) + (H+k) 
\scal{\partial_j \left( h_{im} g^{ml} \partial_l F \right)}{\nu} \\\nonumber
& =& \partial^2_{ij} (H+k) -\Gamma_{ij}^k \partial_k (H+k) + (H+k)h_{im} g^{ml}\scal{  \Gamma_{lj}^k \partial_k F - h_{lj} \nu}{\nu} 
\\\label{eqnu}
&=& \nablas_i \nablas_j (H+k) - (H+k) h_{il} g^{lm} h_{mj}.
\end{eqnarray}
Using Codazzi's equations, one can show that
\begin{equation}\Deltas h_{ij} = \nablas_i \nablas_j H 
+ H h_{il}g^{lm}h_{mj} - |A|^2 h_{ij}, \label{dhij} \end{equation}
so that \eqref{secondevol} follows from \eqref{dhij} and \eqref{eqnu}.
%\begin{equation}
%  \frac{d}{dt} h_{ij} = \Deltas h_{ij} +\nablas_i \nablas_j g-2H h_{il}g^{lm}h_{mj} -g g^{ml} h_{im} h_{jl} +|A|^2 h_{ij}. \label{2FF}
%\end{equation}
%
{}From \eqref{secondevol} we deduce
\begin{equation*}
\begin{aligned}
 \frac{d}{dt} H &= \frac{d}{dt} g^{ij} h_{ij} \\&= 2(H+k)g^{is}h_{sl} g^{lj} h_{ij} +g^{ij} \left(\nablas_i \nablas_j (H+k) - (H+k) h_{il} g^{lm} h_{mj} \right) \\
&=\Deltas (H+k) +(H+k) |A|^2 ,
\end{aligned}
%\label{mcevol}
\end{equation*}
which gives \eqref{mcevol}.
%\begin{equation}
%\begin{aligned}
% \partial_i (g^{lm}) &= g^{lu} \partial_i(g_{uv}) g^{vm} = g^{lu} \left( \scal{\partial_{iu} F}{\partial_v F} + \scal{\partial_u F}{\partial_{iv} F} \right) g^{vm} \\
%&= g^{lu} \left( \scal{ \Gamma_{iu}^k  \partial_k F - h_{iu} \nu}{\partial_v F} + \scal{\partial_u F}{\Gamma_{iv}^k \partial_k F - h_{iv} \nu} \right) g^{vm} \\
%&= g^{lu} (\Gamma_{iu}^k g_{kv} +\Gamma_{iv}^k g_{uk}) g^{vm} = g^{lu} \Gamma_{iu}^m + g^{vm} \Gamma_{iv}^l.
%\end{aligned}
%\end{equation}
In addition, we get
\begin{equation}\label{eqB}
\begin{aligned}
\frac{d}{dt} |A|^2 &= \frac{d}{dt} \left( g^{ik}g^{jl} h_{ij} h_{kl} \right) \\
&= 2 \frac{d}{dt} g^{jl} h_{ij} h_{kl} + 2 g^{ik} g^{jl} \frac{d}{dt} h_{ij} h_{kl} \\
&= 2 \left( 2(H+k)g^{js} h_{st} g^{tl}\right) g^{jl} h_{ij} h_{kl} \\
& \quad + 2 g^{ik} g^{jl} \left(\Deltas h_{ij} +\nablas_i \nablas_j k-2H h_{il}g^{lm}h_{mj} -k g^{ml} h_{im} h_{jl} +|A|^2 h_{ij} \right) h_{kl} \\
&= 2k g^{js} h_{st} g^{tl} g^{jl} h_{ij} h_{kl}+2 g^{ik} g^{jl} \Deltas h_{ij}h_{kl} + 2|A|^4 + 2 \scalv{A}{(\nablas)^2 k}.
\end{aligned}
\end{equation}
On the other hand, one has 
\begin{equation}\label{eqC}
 \Deltas |A|^2 = 2 \scalv{\Deltas A}{A} + 2 |\nablas A|^2 = 2g^{pq} g^{mn} h_{pm} \Deltas h_{qn} +  2 |\nablas A|^2.
\end{equation}
so that \eqref{evoln2FF} follows from \eqref{eqC} and \eqref{eqB}.
%\begin{equation}
%\begin{aligned}
% \frac{d}{dt} |A|^2 &= \Deltas |A|^2 +2 g C + 2|A|^4 - 2 |\nablas A|^2+ 2 \scalv{A}{(\nablas)^2 g}.
%\end{aligned}
%\label{evoln2FF}
%\end{equation}
%where $C =g^{ij}g^{kl}g^{mn} h_{ik} h_{lm} h_{nj}.$ 
%In normal coordinates, note that $C =  \sum_{i} \lambda_i^3$.
\end{proof}

%\subsection{No loss of embededness in short time}

%First, let us state a very well known comparison principle. %Quote?
%\begin{prop}
% Let $M$ and $N$ two $C^{1,1}$ hypersurfaces whose evolutions by mean curvature with (smooth) forcing term $g$ are respectively denoted by $M_t$ and $N_t.$ Assume that $M \subset N.$ Then, for all time the two evolutions exist, we have $M_t \subset N_t.$
%\label{compg}
%\end{prop}
%One can use this principle and the sphere conditions to control the motion using the

\subsection{The Monotonicity Formula}

We extend Huisken's monotonicity formula \cite{huisken89} 
to the forced mean curvature flow \eqref{evolg}
(see also \cite[Section 2.2]{cesnovval11}). 

Given a vector field $\omega:M_t\to \R^{d+1}$, we let
\[
\omega^\perp = \scal{\omega}{\nu}\nu,\qquad
\omega^T=\omega-\omega^\perp\,.
\]
Letting $X_0\in\R^{d+1}$ and $t_0\in \R$,
%a fixed point of the moving surface $M_t$. 
for $(X,t) \in \R^{d+1}\times [t_0,+\infty)$ we define the kernel
 $$
\rho(X,t)=\frac{1}{(4\pi(t_0-t))^{d/2}} \exp\left( \frac{-|X_0-X|^2}{4(t_0-t)}\right).
$$ 
A direct computation gives
\begin{equation}\frac{d \rho}{d t} = - \Deltas \rho + \rho \left( \frac{\scal{X_0-X}{(H+k) \nu}}{t_0-t} - \frac{|(X_0-X)^\perp |^2}{4(t_0-t)^2} \right).\label{kernel}\end{equation}

\begin{prop}[Monotonicty Formula]
$$
\frac{d}{dt} \int_{M_t} \rho = - \int_{M_t} \rho \left( \left \vert H + \frac{k}{2} + \frac{\scal{X-X_0}{\nu}}{2(t_0-t)} \right \vert^2 - \frac{k^2}{4}\right).
$$ 
\end{prop}

\begin{proof}
Recalling \eqref{eqmeas}, we compute

 \begin{align*}
  \frac{d}{dt} \int_{M_t} \rho &= \int_{M_t} \frac{d}{dt}\rho - H(H+k) \rho \\
  &= \int_{M_t} \rho \left( - \frac{|X-X_0|^2}{4(t_0-t)^2} + \frac{d}{2(t_0-t)} - \frac{\scal{X-X_0}{\nu}}{2(t_0-t)} (H+k) - H(H+k) \right) \\
  &= - \int_{M_t} \rho \left( \left \vert H \nu+ \frac{X-X_0}{2(t_0-t)} 
+ \frac{k\nu}{2} \right \vert^2 - \frac{k^2}{4} \right) 
+ \int_{M_t} \frac{d}{2(t_0-t)} \rho 
+ \int_{M_t} \rho \frac{\scal{X-X_0}{\nu} H}{2(t_0-t)} \\
\end{align*}
We use the first variation formula: for all vector field $\mathbf Y$ on $M_t$, we have
$$\int_{M_t} \div_{M_t} \mathbf Y = \int_{M_t} \scalv{H \nu}{\mathbf Y}.$$
As a result, with $\mathbf Y = \frac{\rho (X-X_0)}{2(t-t_0)}$, we get
\begin{align*}
  \frac{d}{dt}  \int_{M_t} \rho &= - \int_{M_t} \rho \left( \left \vert H \nu +\frac{X-X_0}{2(t_0-t)} 
+ \frac{k\nu}{2} \right \vert^2 - \frac{k^2}{4} 
- \frac{|(X-X_0)^T |^2}{4(t_0-t)^2} \right) \\
  &= -\int_{M_t} \rho \left( \left \vert H +\frac{\scal{X-X_0}{\nu}}{2(t_0-t)} + \frac{k}{2} \right \vert^2 - \frac{k^2}{4} \right).
 \end{align*}
\end{proof}

In a similar way (see \cite{ecker89}) one can prove that for all functions $f(X,t)$ defined on $M_t$, one has
\begin{equation}
\partial_t \int_{M_t} \rho f = \int_{M_t} \left( \frac{d f}{dt} - \Deltas f\right) \rho - \int_{M_t} f\rho \left(\left \vert H + \frac{\scal{X-X_0}{\nu}}{2(t_0-t)} 
+ \frac{k}{2} \right \vert^2- \frac{k^2}{4}\right).
 \label{monotonf}
\end{equation}
Indeed, using \eqref{kernel}
\begin{align*}
 \frac{d}{dt} \int_{M_t} \rho f &= \int_{M_t} f \frac{d\rho}{dt} + \frac{df}{dt} \rho - H(H+k) f\rho \\
 &= \int_{M_t} f\left( \frac{d\rho}{dt}  - H(H+k) \rho \right) + \frac{df}{dt} \rho \\
 &= \int_{M_t} f\left( - \Deltas \rho + \rho \left( \frac{\scal{X_0-X}{(H+k) \nu}}{t_0-t} - \frac 14 \frac{|(X_0-X)^\perp |^2}{(t_0-t)^2} \right) - H(H+k) \rho \right) + \frac{df}{dt} \rho \\
 &= \int_{M_t} -\Deltas f \rho +\left( \rho \left( \frac{\scal{X_0-X}{(H+k) \nu}}{t_0-t} - \frac 14 \frac{|(X_0-X)^\perp |^2}{(t_0-t)^2} \right) - H(H+k) \rho \right) + \frac{df}{dt} \rho \\
 &= \int_{M_t} \rho\left( \frac{d}{dt} f -\Deltas f \right) 
- \int f \rho \left( \left \vert  H + \frac{\scal{X-X_0}{\nu}}{2(t_0-t)} 
+ \frac{k}{2} \right \vert^2 - \frac{k^2}{4} \right).
\end{align*}

%This computation yields to the 
\begin{lem}
 Let $f$ be defined on $M_t$ and satisfy
 \begin{equation}\label{eqass}
\frac{d}{dt} f - \Deltas f \ls a \cdot \nablas f \quad \text{on }M_t
\end{equation}
 for some vector field $a$ bounded on $[0,t_1]$. Then, 
 $$\sup_{M_t, \, t\in [0,t_1]} f \ls \sup_{M_0} f.$$
 \label{lemh}
\end{lem}
\begin{proof}
 Denote by $a_0$ the bound on $a$, $k:=\sup_{M_0} f$ and define $f_l = \max (f-l,0).$ Assumption \eqref{eqass} implies
 $$\left( \frac{d}{dt} - \Deltas \right) f_l^2 \ls 2 f_l a \cdot \nablas f_l - 2 |\nablas f_l |^2$$ which, thanks to Young's inequality, gives
 $$\left( \frac{d}{dt} - \Deltas \right) f_l^2 \ls \frac{1}{2} a_0^2 f_l^2.$$
 Applying \eqref{monotonf} to $f_l^2$, we get
\begin{equation}\label{eqfk}
\frac{d}{dt} \int f_l^2 \rho \ls \frac{1}{2} (a_0^2+\Vert k \Vert_\infty^2)  \int f_l^2 \rho. 
\end{equation}
Letting $l=\sup_{M_0}f$, so that $f_l\equiv 0$ on $M_0$, 
from \eqref{eqfk} and the Gronwall's Lemma we obtain 
that $f_l\equiv 0$ on $M_t$ for all $t\in (0,t_1]$, 
which gives thesis.
\end{proof}

\section{Proof of Theorem \ref{main}}\label{geometric}

We now prove short time existence  
for the mean curvature flow with obstacles \eqref{eqmo}, \eqref{eqcon}. 
Let $M_0=\partial E(0)\subset U$, where we assume that $U$, $E(0)$ 
are open sets with boundary uniformly of class $C^{1,1}$, with
${\rm dist}(M_0,\partial U)>0$.
In particular, $M_0$ satisfies a uniform exterior and 
interior ball condition, that is, there is $R>0$ 
such that, for every $x \in M_0$, 
one can find two open balls $B^+$ and $B^-$ of radius $R$ 
which are tangent to $M_0$ at $x$ and such that $B^+ \subset E(0)^c$ and 
$B^- \subset E(0).$  
Let also $\Omega^-:=E(0)\setminus \overline U$, 
and $\Omega^+:=E(0)^c\setminus(M_0\cup\overline U)$. 
Notice that $\Omega^\pm$ are 
open sets with $C^{1,1}$ boundaries, with ${\rm dist}(\Omega^-,\Omega^+)>0$. 
%$\Omega^-\subset \Omega$ and $ \Omega^+ \subset \Omega^c$. 
Let also $$k := 2N( \chi_{\Omega^-} - \chi_{\Omega^+}) $$
where $N$ is bigger than ($d$ times) the mean curvature of $\partial U.$

We want to show that equation \eqref{evolg}, with $k$ as above, 
has a solution in an interval $[0,T).$ 
To this purpose, letting
$\rho_\varepsilon$ be a standard mollifier supported in the ball
of radius $\varepsilon$ centered at $0$, 
we introduce a smooth regularization $\gep = k \ast \rho_\varepsilon$ of $k$. 
Notice that $\Vert \gep \Vert_\infty = 2N$, 
$\gep(x) = 2N$ (resp. $\gep(x) = -2N$) at every $x\in \Omega^-$ 
(resp. $x\in \Omega^-$) such that ${\rm dist}(x,\partial U)\ge\varepsilon$,
and $\gep(x) = 0$ at every $x\in U$ such that 
${\rm dist}(x,\partial U)\ge\varepsilon$.

%\subsection{Short time existence of the smooth motion}

Using standard arguments (see for instance \cite[Theorem 4.1]{ecker91b} and \cite[Prop. 4.1]{ecker91}) 
one can show existence of a smooth solution $M^\varepsilon_t$ of \eqref{evolg},
with $k$ replaced by $\gep$, on a maximal time interval $[0,T_\varepsilon).$ 

Let now
\[
\Om^\pm_\varepsilon := \{x\in\Om^\pm:\ {\rm dist}(x,U)>\varepsilon\}.
\]
The following result follows directly from the definition of $\gep$.

\begin{prop}
The hypersurfaces $\partial \Omega^\pm_\varepsilon$ are respectively a super and a subsolution of \eqref{evolg}, with $k$ replaced with $k_\varepsilon$. 
In particular, by the parabolic comparison principle $M^\varepsilon_t$ cannot intersect $\partial \Omega^\pm_\varepsilon$.
\label{propob}
\end{prop}

We will show that we can find a time $T >0$ 
such that for every $\varepsilon,$ there exists 
a smooth solution of \eqref{evolg} (with $k$ replaced with $k_\varepsilon$) on $[0,T).$

The following result will be useful in the sequel. We omit the proof which is
a simple ODE argument.

\begin{lem}\label{lemball}
Let $M_0=\partial B_R(x_0)$ be a ball of radius $R\le 1$ centered at $x_0$. Then, the
evolution $M_t$ by \eqref{evolg}, with constant forcing term $k=2N$, is given by
$M_t=B_{R(t)}(x_0)$ with $R(t) \gs \sqrt{R^2 -(4N+2d)t}$. In particular, 
the solution exists at least on $\left[0,\frac{R^2}{4N+2d}\right).$
\end{lem}

\begin{prop}\label{propballs}
There exists $r>0$, a collection of balls $B_i=B_r(x_i)$ 
of radius $r$, and a positive time $T_0$ such that 
$M_t^\varepsilon \subset \bigcup_i B_i$ for every $t \in [0,\min(T_0,T_\varepsilon))$.
In addition, we can choose the balls $B_i$ in such a way that, for every $i$, 
there exists $\omega_i\in\R^{d+1}$ such that $\partial \Omega^\pm \cap B_{4r}(x_i)$ are graphs of some functions $\psi_i^\pm : \R^d \to \R \cup \{\pm \infty\}$ over $\omega_i^\perp$.\\
In particular, one has $$\scal{\nabla k_\varepsilon}{\omega_i} 
\gs |\nabla k_\varepsilon|/2 \quad \text{on } B_{2r}(x_i).$$
\end{prop}
Most of these notations are summarized in Figure \ref{figbal}.
\begin{figure}
\centering
\includegraphics{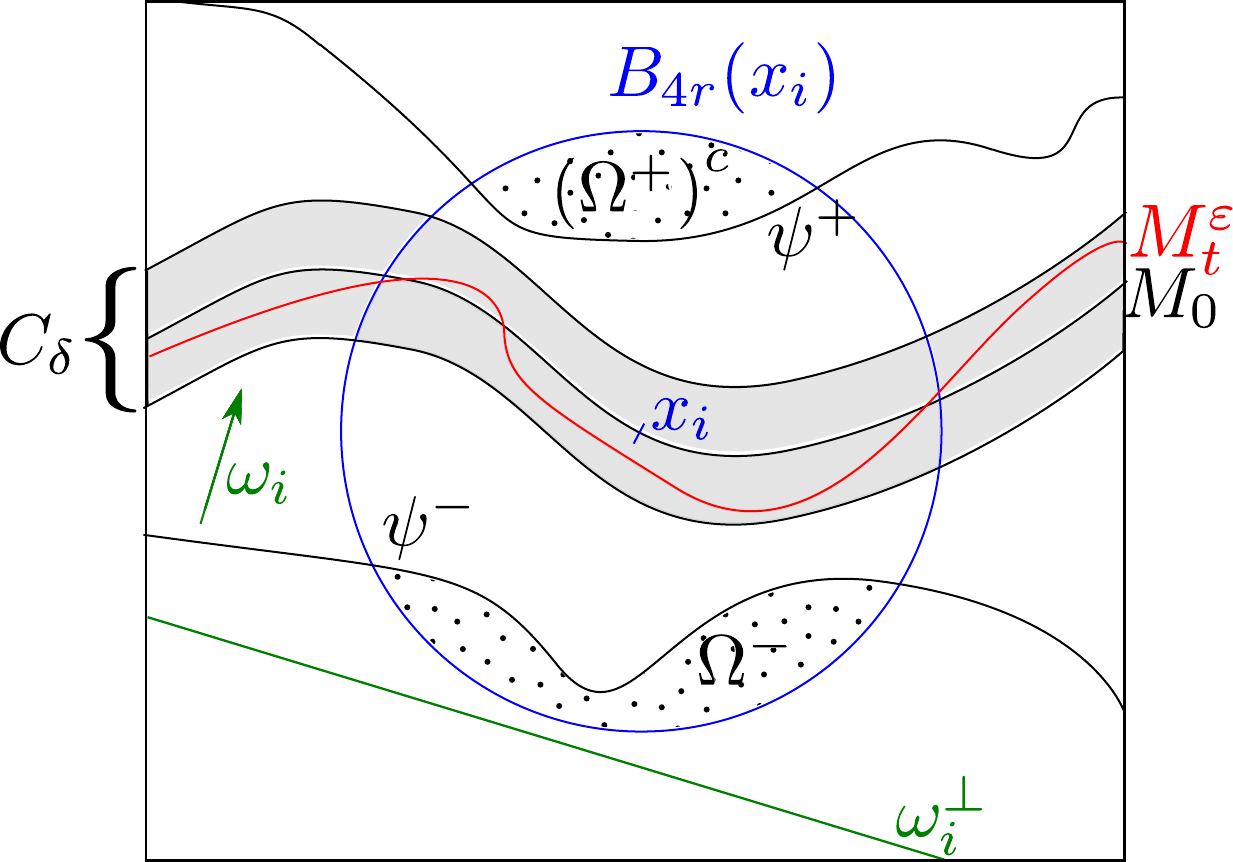}
\caption{Main notations of Proposition \ref{propballs}.}
\label{figbal}
\end{figure}

\begin{proof}
By assumption, for every $\x \in M_0$ there exist interior 
and exterior balls $B_x^\pm$ of fixed radius $R\le 1$. 
Let $B_x^\pm(t)$ be the evolution of $B_x^\pm$ by \eqref{evolg} with forcing term $k=2N$.
By comparison, for every $t \in [0,T_\varepsilon)$, 
$B^+_x(t) \subset \Omega(t)^c$ and $B^-_x(t) \subset \overline \Omega(t).$ 
Recalling Lemma \ref{lemball}, there exists $\delta >0$ and $T_0 >0$, independent of $\varepsilon$, such that 
$M_t \subset \{d_{M_0} \ls \delta\}=: C_\delta$,
for all $t\in [0,\min(T_\varepsilon, T_0))$. \\
We eventually reduce $\delta, T_0$ such that $C_\delta$ can be covered with a collection of balls $B_i=B_r(x_i)$, 
centered at $x_i\in M_0$ and with a radius $r$ such that, for every $i$, there exists a unit vector $\omega_i\in\R^{d+1}$ satisfying
$$\scal{\omega_i}{\nu^+(x)} \gs \frac{1}{2} \quad \text{ and } \quad \scal{\omega_i}{\nu^-(y)} \gs \frac 12$$
for every $x \in \partial \Omega^+ \cap B_{4r}(x_i)$ and $y \in \partial \Omega^-\cap B_{4r}(x_i)$, where $\nu^\pm$ is the outer normal to $\Omega^\pm.$

As a result, $\partial \Omega^\pm \cap B_{4r}(x_i)$ are graphs of some functions $\psi_i^\pm : \R^d \to \R \cup \{\pm \infty\}$ over $\omega_i^\perp$ (see Figure \ref{figbal}).

Notice also that $k$ is a $BV$ function and $Dk$ is a Radon measure concentrated on $\partial U$ such that 
%where $d\sigma^\pm$ denote the measure of the area on $\partial \Omega^\pm$, 
 $$\scal{D k}{\omega_i} \gs \frac{|D k|}{2} \text{ on }B_{4r}(x_i).$$

Then, for every $x\in B_{2r}(x_i)$ and $\eps$ sufficiently small (such that $\rho_\eps(x) = 0$ as soon as $|x| \gs 2r$), we have
\begin{align*}
\scal{\nabla k_\varepsilon}{\omega_i} & = 
\scal{\nabla \int_{\R^{d+1}} k(x-y) \rho_\varepsilon(y) dy}{\omega_i} \\
&=\int_{\R^{d+1}} \scal{D k(x-y)}{\omega_i}  \rho_\varepsilon(y) dy \\
& \gs \int_{\R^{d+1}} \frac{|D k| (x-y)}{2} \rho_\varepsilon(y) dy \\
%\text{ because $x-y \in 2B_i + \operatorname{supp} (\rho_\varepsilon) \subset 3B_i$ } \\
& \gs \frac{|D k| \ast \rho_\varepsilon}{2} \gs \frac{|\nabla k_\varepsilon|}{2}.
\end{align*}
%which was expected.
\end{proof}

In what follows, we will control the geometric quantities of $M_t^\varepsilon$ 
inside each ball $B_i.$ As in \cite{ecker91}, we introduce a localization function $\phi_i$ as follows:
let $\eta_i(x,t)= |x-x_i|^2 +(2d+\Lambda)t$ ($\Lambda$ be a positive constant that will be fixed later) and, for $R=2r$, $\phi_i(x,t)=(R^2-\eta_i(x,t))^+.$ 
We denote by $\phi_i$ the quantity $\phi_i(\x,t)$, where $\x=\x(p,t)$ will be a generic point in $M_t$. 
Notice that there exists $T_1 =\frac{r^2}{2d+\Lambda}$ such that for all $t \in [0,\min(T_1,T_\varepsilon))$, 
\begin{equation}
 M_t^\varepsilon \subset \bigcup_i \{\phi_i > r^2 \}.
\label{cover}
\end{equation}

As a result, we have the following
\begin{lem}
Let $f$ be a smooth function defined on $M_t^\varepsilon$. Assume that there is a $C >0$ such that 
$$ 
\phi_i f \ls C \text{ on }M_t^\varepsilon \qquad
 \forall t \ls \min(T_\varepsilon, T_1) \text{ and } \forall i \in \mathbb N .
$$
Then, 
$$
f \ls  \alpha C \text{ on }M_t^\varepsilon \qquad
\forall t \ls \min(T_\varepsilon,T_1),$$
where $\alpha$ depends only on the $C^{1,1}$ norm of $M_0$.
\label{lemloc}
\end{lem}

\begin{lem}\label{lemcgrad}
Let $v:= \scal{\nu}{\omega}^{-1}$.
The quantity $v^2 \phi^2$ satisfies
\begin{eqnarray}\nonumber
\left(\frac{d}{dt} - \Deltas \right) \left( \frac{v^2 \phi^2}{2} \right) &\ls&  \frac{1}{2}   \scal{\nablas (v^2 \phi^2)}{\frac{\nablas \phi^2}{\phi^2}} 
\\\label{last}
&&+ \phi^2 v^3 \scal{\nablas k_\varepsilon}{\omega}+ v^2\phi(2 k_\varepsilon \scal \x \nu-\Lambda).
\end{eqnarray}
\end{lem}

\begin{proof}
In this proof and the proofs further, we use normal coordinates: we assume that $g_{ij} ) \delta_{ij}$ (Kronecker symbol) and that the Christoffel symbols $\Gamma_{ij}^k$ vanish at the computation point.

We expand the derivatives
\begin{eqnarray*}
 \left(\frac{d}{dt} - \Deltas \right) \left( \frac{v^2 \phi^2}{2} \right) & = & v^2 \left(\frac{d}{dt} - \Deltas \right) \frac{\phi^2}{2} +  \phi^2 \left(\frac{d}{dt} - \Deltas \right) \frac{v^2}{2} 
- 2 \scalv{\nablas \frac{\phi^2}{2}}{\nablas \frac{v^2}{2}}.
\end{eqnarray*}
\emph{First term.} We start computing
$$\left(\frac{d}{dt} - \Deltas \right) |\x|^2 =- 2k_\varepsilon \scal \x \nu - 2d.$$
Then,
$$\left(\frac{d}{dt} - \Deltas \right) \phi^2 = 2\phi(2 k_\varepsilon \scal{\x-x_i}{\nu}-\Lambda) - 2 |\nablas |\x|^2|^2.$$
\emph{Second term.} We are interested in
\begin{eqnarray}
 \frac{1}{2} \frac{d}{dt} \scal{\omega}{\nu}^2 &=& \scal{\omega}{\nu} \scal{\frac{d}{dt} \nu}{\omega} \\
 &= & \scal{\omega}{\nu} \scal{\nablas (H+k_\varepsilon)}{\omega}.
\end{eqnarray}
So,
\begin{equation}
 \frac{1}{2} \frac{d}{dt} \scal{\omega}{\nu}^{-2} = - \scal{\omega}{\nu}^{-3} \scal{\nablas (H+k_\varepsilon)}{\omega}.
\end{equation}
On the other hand,

\begin{align}
 \frac 12 \Deltas(\scal{\omega}{\nu}^{-2}) &=\scal{\omega}{\nu}^{-1} \Deltas \scal{\omega}{\nu}^{-1} - \scalv{\nablas \scal{\omega}{\nu}^{-1}}{\nablas \scal{\omega}{\nu}^{-1}}.
\end{align}
Let us note that
\begin{align*}
 \partial_{ij} \nu = \partial_i\left( h_{jl} g^{lm} \partial_m F \right) = \partial_i(h_{jl}) \delta_{lm} \partial_m F - h_{jl} \delta_{lm} (-h_{im} \nu) 
=\partial_i (h_{jl}) \partial_l F - \lambda_i^2 \delta_{ij} \nu.
\end{align*}
We then get
\begin{align}
 \Deltas \scal{\omega}{\nu}^{-1} &= \partial_{ii} \scal{\omega}{\nu}^{-1} = \partial_i \left(- \scal{\omega}{\partial_i \nu} \scal{\omega}{\nu}^{-2} \right)
\\ &= -\scal{\omega}{\partial_{ii}\nu} \scal{\omega}{\nu}^{-2}+2 \scal{\omega}{\partial_i \nu}^2 \scal{\omega}{\nu}^{-3} \\
&=   -\scal{\omega}{\nu}^{-2} \scal{ \partial_{i} h_{il} \partial_l F - \lambda_i^2 \nu}{\omega} +2\scal{\omega}{\nu}^{-3} \scal{\omega}{ \lambda_i \partial_i F}^2. \\
&=- \scal{\omega}{\nu}^{-2} \scal{ \partial_{l} h_{ii} \partial_l F}{\omega} + |A|^2 \scal{\nu}{\omega}^{-1}   +2\scal{\omega}{\nu}^{-3} \scal{\omega}{ \lambda_i \partial_i F}^2.  %Weingarten
\end{align}
We also have 
\begin{align}
\scalv{\nablas \scal{\omega}{\nu}^{-1}}{\nablas \scal{\omega}{\nu}^{-1}} &= \scal{\omega}{\nu}^{-4} \scal{\omega}{\partial_k \nu} \scal{\omega}{\partial_k \nu} \\
&= \scal{\omega}{\nu}^{-4} \scal{\omega}{h_{ku} g^{uv} \partial_v F}^2 = \scal{\omega}{\nu}^{-4} \scal{\omega}{\lambda_k \partial_k F}^2,
\end{align}
which leads to
\begin{eqnarray*}
 \left(\frac{d}{dt} - \Deltas \right) \frac{v^2}{2} &=& -v^3 \scal{\nablas (H+k_\varepsilon)}{\omega} +v^3 \partial_m(h_{ii})\scal{\omega}{\partial_m F}\\
 && - |A|^2 v^2 -2 v^4 \lambda_k^2 \scal{\omega}{ \partial_k F}^2 - v^4 \scal{\omega}{\lambda_k \partial_k F}^2
\end{eqnarray*}
\emph{Third term.} We notice, as in \cite{ecker91} that $|\nablas \phi^2|^2 = 4\phi^2 | \nablas (|\x|^2)|^2$ and
$$-\scal{ \nablas(v^2) }{\nablas \phi^2} = -3 \scal{v\nablas (v)}{\nablas \phi^2} + \frac{1}{2}  \left( \scal{\nablas (v^2 \phi^2)}{\frac{\nablas \phi^2}{\phi^2}} -v^2 \frac{\vert \nablas \phi^2 \vert^2}{\phi^2} \right).$$ 
Then, Young's inequality gives
\begin{align*}2 \left \vert v \scal{\nablas v}{\nablas \phi^2} \right \vert &\ls 2 \phi^2 \vert \nablas v^2 \vert^2 + \frac{1}{2\phi^2} \vert \nablas \phi^2 \vert^2 \\&\ls 2 \phi^2 \vert \nablas v^2 \vert^2 + 2 v^2 \vert \nablas |\x|^2\vert^2.
\end{align*}
Hence,
$$-\scal{ \nablas(v^2) }{\nablas \phi^2} \ls  -3 \phi^2 \vert \nablas v^2 \vert^2 -3 v^2 \vert \nablas |\x|^2\vert^2+ \frac{1}{2}  \left( \scal{\nablas (v^2 \phi^2)}{\frac{\nablas \phi^2}{\phi^2}} -v^2 \frac{\vert \nablas \phi^2 \vert^2}{\phi^2} \right).$$ 

Summing the three terms, we get

\begin{equation*}\left(\frac{d}{dt} - \Deltas \right) \left( \frac{v^2 \phi^2}{2} \right) \ls  \frac{1}{2}   \scal{\nablas (v^2 \phi^2)}{\frac{\nablas \phi^2}{\phi^2}} - \phi^2 v^3 \scal{\nablas k_\varepsilon}{\omega}+ v^2\phi(2 k_\varepsilon \scal \x \nu-\Lambda).
\end{equation*}
\end{proof}

For $\gamma>0$, we let
$$
\psi (v^2 ) := \frac{\gamma v^2}{1-\gamma v^2}.
$$

\begin{lem}\label{lemc2ff}
For $\varepsilon \ls r$, we have
\begin{eqnarray*}
\left( \frac{d}{dt} - \Deltas \right) \frac{\phi^2 |A|^2 \psi(v^2)}{2} &\ls & \phi^2 \psi(v^2)(-\gamma |A|^4 -2 k_\varepsilon \sum_i \lambda_i^3 
- 2 \scalv{A}{(\nablas)^2 k_\varepsilon} )\\
 && - \phi^2 |A|^2 v^3 \psi'(v^2) \scal{\nablas k_\varepsilon}{\omega} -\phi^2 |A|^2 \sum_i (\lambda_i \omega^i)^2\frac{2v^4+\gamma v^6}{(1-\gamma v^2)^3}. 
\end{eqnarray*}
\end{lem}

\begin{proof}

We denote $V = \frac{\phi^2 |A|^2 \psi(v^2)}{2}$ and compute

\begin{eqnarray*}
 \left( \frac{d}{dt} - \Deltas \right) \frac{\phi^2 |A|^2 \psi(v^2)}{2} & = & |A|^2 \psi(v^2) \left( \frac{d}{dt} - \Deltas \right) \frac 12 \phi^2 + \phi^2 \psi(v^2) \left( \frac{d}{dt} - \Deltas \right) \frac 12 |A|^2\\ && + \phi^2 |A|^2 \left( \frac{d}{dt} - \Deltas\right) \frac 12 \psi(v^2) 
 -2\scalv{1/2 \nablas |A|^2 }{1/2 \nablas \phi^2} \\ && - 2\scalv{1/2 \nablas |A|^2}{1/2\nablas \psi(v^2) } -2 \scalv{1/2\nablas \phi^2}{1/2\nablas \psi(v^2)}.
\end{eqnarray*}
The two first terms have already been computed. Let us consider in the third one.
$$\frac 12 \frac{d}{dt} \psi(v^2) = v \frac{d v}{d t} \psi'(v^2) = -v^3 \psi'(v^2) \scal{\nablas (H+k_\varepsilon)}{\omega},$$
\begin{align*}\frac 12 \Deltas \psi(v^2) &= \frac 12 \partial_{ii} \psi(v^2) = \partial_i (v\partial_i v\psi'(v^2)) = v\Deltas v \psi'(v^2) + 2v^2 |\nablas v|^2 \psi''(v^2) + |\nablas v|^2 \psi'(v^2)\\
& =  (3 |\nablas v|^2 -v^3 (\partial_l (h_{kk}) w^l)+v^2|A|^2)\psi'(v^2) +2|\nablas v|^2 \psi''(v^2).
\end{align*}
Hence
$$\left( \frac{d}{dt} - \Deltas\right) \frac 12 \psi(v^2) = -v^3 \psi'(v^2) \scal{\nablas k_\varepsilon}{\omega} - (3 |\nablas v|^2 +v^2|A|^2)\psi'(v^2) - 2 v^2|\nablas v|^2 \psi''(v^2).$$

As above, we want to conclude the proof using the weak maximum principle. So, we want to rewrite the last terms (which are gradient terms) using the gradient of $V$. Let us expand $\nablas V$.
\begin{align*}
 \nablas \frac{\phi^2 |A|^2 \psi(v^2)}{2} &= \phi^2 |A|^2 \frac 12 \nablas \psi(v^2) + |A|^2 \psi(v^2) \frac 12\nablas \phi^2 + \phi^2 \psi(v^2) \frac 12 \nablas |A|^2.
\end{align*}
So,
\begin{eqnarray*}
 \left \vert \nablas \frac{\phi^2 |A|^2 \psi(v^2)}{2} \right \vert^2 & = & \phi^4 |A|^4 \frac{|\nablas \psi(v^2)|^2}{4} + |A|^4 \psi^2(v^2) \frac{|\nablas \phi^2|^2}{4} + \phi^4 \psi^2(v^2) \frac{|\nablas |A|^2|^2}{4}\\
 & & +\phi^2 |A|^4 \psi(v^2) \scalv{\nablas \psi(v^2)}{\nablas \phi^2} + \phi^4 |A|^2 \psi(v^2) \scalv{\nablas \psi(v^2)}{\nablas |A|^2} \\ &&+ |A|^2 \psi^2(v^2) \phi^2 \scalv{\nablas \phi^2}{\nablas |A|^2}.
\end{eqnarray*}
As a matter of fact,
\begin{eqnarray*}
\frac{1}{\phi^2 |A|^2 \psi(v^2)} \left \vert \nablas \frac{\phi^2 |A|^2 \psi(v^2)}{2} \right \vert^2 &=& \phi^2 |A|^2 \frac{|\nablas \psi(v^2)|^2}{4\psi(v^2)} + |A|^2 \psi(v^2) \frac{|\nablas \phi^2|^2}{4\phi^2} \\ 
&& + \phi^2 \psi(v^2) \frac{|\nablas |A|^2|^2}{4|A|^2} 
 +2|A|^2 \scalv{\nablas \psi(v^2)/2}{\nablas \phi^2/2} \\ 
&& + 2\phi^2 \scalv{\nablas \psi(v^2)/2}{\nablas |A|^2/2} +2 \psi(v^2)\scalv{\nablas \phi^2/2}{\nablas |A|^2/2}.
\end{eqnarray*}
We use the last equality to rewrite
\begin{equation}
\begin{aligned}
 \left( \frac{d}{dt} - \Deltas\right)& \frac{\phi^2 |A|^2 \psi(v^2)}{2}  \\
& = |A|^2 \psi(v^2) \left( \phi(2 k_\varepsilon \scal{\x}{\nu}-\Lambda) -  |\nablas |x|^2|^2 \right) \\ 
&+ \phi^2 \psi(v^2) \left( -  \scalv{\nablas A}{\nablas A} + |A|^4 -2k_\varepsilon g^{js} h_{st} g^{tl} g^{jl} h_{ij} h_{kl} - 2 \scalv{A}{\nabla^2 k_\varepsilon} \right) \\
 & + \phi^2 |A|^2 \left( -v^3 \psi'(v^2) \scal{\nablas k_\varepsilon}{\omega} - (3 |\nablas v|^2 +v^2|A|^2)\psi'(v^2) - 2v^2 |\nablas v|^2 \psi''(v^2) \right)  \\
& -\frac{1}{\phi^2 |A|^2 \psi(v^2)} \left \vert \nablas \frac{\phi^2 |A|^2 \psi(v^2)}{2} \right \vert^2 \\
&+\phi^2 |A|^2 \frac{|\nablas \psi(v^2)|^2}{4\psi(v^2)} + |A|^2 \psi(v^2) \frac{|\nablas \phi^2|^2}{4\phi^2} + \phi^2 \psi(v^2) \frac{|\nablas |A|^2|^2}{4|A|^2}.
\end{aligned}
\label{c2ff}
\end{equation}

Let us precise some terms: 
\begin{eqnarray*}
 |\nablas \phi^2|^2 &=& 4\phi^2 \cdot | -2 \x^T|^2 = 4\phi^2(4|\x|^2 - 4\scal{\x}{\nu}), \\
|\nablas \psi(v^2)|^2 &=& \psi'(v^2)^2 |\nablas v^2|^2 = 4\psi'(v^2)^2 v^6 \sum_k (\lambda_k \omega^k)^2, \\
|\nablas |A|^2 |^2 &=& 4 \sum_i (\partial_i (h_{ll}) \lambda_l)^2, \\
|\nablas A|^2 &=& \sum_{i,k,l} (\partial_i(h_{km}))^2.
\end{eqnarray*}

In addition, we have the obvious estimate
$$ |\nablas |A|^2 |^2 \ls 4 |A|^2 |\nablas A|^2.$$
So,
\begin{multline*}
 \phi^2 |A|^2 \frac{|\nablas \psi(v^2)|^2}{4\psi(v^2)} + |A|^2 \psi(v^2) \frac{|\nablas \phi^2|^2}{4\phi^2} + \phi^2 \psi(v^2) \frac{|\nablas |A|^2|^2}{4|A|^2} \\
\ls \phi^2 |A|^2 \frac{\psi'(v^2)^2 v^6 \sum_k (\lambda_k \omega^k)^2}{\psi(v^2)} +4 |A|^2 \psi(v^2) (|\x|^2 - \scal{\x}{\nu}^2) + \phi^2 \psi(v^2) |\nablas A|^2.
\end{multline*}

We plug this inequality into \eqref{c2ff} and obtain

\begin{align*}
 \left( \frac{d}{dt} - \Deltas\right) \frac{\phi^2 |A|^2 \psi(v^2)}{2} & \ls |A|^2 \psi(v^2) \left( \phi(2 k_\varepsilon \scal{\x}{\nu}-\Lambda) -  |\nablas |x|^2|^2 \right) \\
 &+ \phi^2 \psi(v^2) \left( -  \scalv{\nablas A}{\nablas A} + |A|^4 -2k_\varepsilon g^{js} h_{st} g^{tl} g^{jl} h_{ij} h_{kl} - 2 \scalv{A}{\nabla^2 k_\varepsilon} \right) \\
 & + \phi^2 |A|^2 \left( -v^3 \psi'(v^2) \scal{\nablas k_\varepsilon}{\omega} - (3 |\nablas v|^2 +v^2|A|^2)\psi'(v^2) - 2v^2 |\nablas v|^2 \psi''(v^2) \right)  \\
& -\frac{1}{\phi^2 |A|^2 \psi(v^2)} \left \vert \nabla \frac{\phi^2 |A|^2 \psi(v^2)}{2} \right \vert^2 \\
&+\phi^2 |A|^2 \frac{\psi'(v^2)^2 v^6 \sum_k (\lambda_k \omega^k)^2}{\psi(v^2)} +4 |A|^2 \psi(v^2) (|\x|^2 - \scal{\x}{\nu}^2) + \phi^2 \psi(v^2) |\nablas A|^2. 
% \\
% &\ls |A|^2 \psi(v^2) \left( \phi(2 g \scal{\x}{\nu}-\Lambda) -  |\nablas |\x|^2|^2 \right) \\ 
% &+ \phi^2 \psi(v^2) |A|^4 -2 \phi^2 \psi(v^2)g C - 2\phi^2 \psi(v^2) \scalv{A}{\nabla^2 g} \\
%  & - \phi^2 |A|^2 v^3 \psi'(v^2) \scal{\nablas g}{\omega} - 3 \phi^2 |A|^2 |\nablas v|^2 \psi'(v^2) - \phi^2 |A|^4 v^2 \psi'(v^2)\\& - 2\phi^2 |A|^2v^2 |\nablas v|^2 \psi''(v^2)  -\frac{1}{\phi^2 |A|^2 \psi(v^2)} \left \vert \nabla \frac{\phi^2 |A|^2 \psi(v^2)}{2} \right \vert^2 \\
% &+\phi^2 |A|^2 \frac{\psi'(v^2)^2 v^6 \sum_k (\lambda_k \omega^k)^2}{\psi(v^2)} + 4|A|^2 \psi(v^2) (|\x|^2 - \scal{\x}{\nu}^2). 
\end{align*}
Let us regroup some terms (noting that $|\nablas v|^2 = v^4 \sum_i (\lambda_i \omega^i)^2$), we get
\begin{align*}
  \left( \frac{d}{dt} - \Deltas\right)& \frac{\phi^2 |A|^2 \psi(v^2)}{2} \\
 &\ls |A|^2 \psi(v^2) \left( \phi(2 k_\varepsilon \scal{\x}{\nu}-\Lambda) \right) \\ 
&+ \phi^2|A|^4( \psi(v^2) -v^2 \psi'(v^2)) -2 \phi^2 \psi(v^2)k_\varepsilon g^{js} h_{st} g^{tl} g^{jl} h_{ij} h_{kl} - 2\phi^2 \psi(v^2) \scalv{A}{\nabla^2 k_\varepsilon} \\
 & - \phi^2 |A|^2 v^3 \psi'(v^2) \scal{\nablas k_\varepsilon}{\omega}-\frac{1}{\phi^2 |A|^2 \psi(v^2)} \left \vert \nablas \frac{\phi^2 |A|^2 \psi(v^2)}{2} \right \vert^2 \\
 & +\phi^2 |A|^2 \sum_i (\lambda_i \omega^i)^2 \left( \frac{v^6 \psi'(v^2)^2}{\psi(v^2)}-3v^4 \psi'(v^2)-2v^6 \psi''(v^2) \right). 
\end{align*}
Then, we note that
$$
\frac{v^6 \psi'(v^2)^2}{\psi(v^2)}-3v^4 \psi'(v^2)-2v^6 \psi''(v^2) 
= -\frac{2v^4+\gamma v^6}{(1-\gamma v^2)^3} \ls 0
$$
and
$$\psi(v^2) - v^2 \psi'(v^2) = - \gamma  \psi^2(v^2) \ls 0.$$
So,
\begin{eqnarray*}
  \left( \frac{d}{dt} - \Deltas\right) \frac{\phi^2 |A|^2 \psi(v^2)}{2} &\ls & \phi^2 \psi(v^2)(-\gamma |A|^4 -2 k_\varepsilon \sum_i \lambda_i^3 - 2 \scalv{A}{\nabla^2 k_\varepsilon} )\\
 && - \phi^2 |A|^2 v^3 \psi'(v^2) \scal{\nablas k_\varepsilon}{\omega} -\phi^2 |A|^2 \sum_i (\lambda_i \omega^i)^2\frac{2v^4+\gamma v^6}{(1-\gamma v^2)^3},
\end{eqnarray*}
what was expected.
\end{proof}

We now show that $M_t$ can be locally written as a Lipschitz graph, with Lipschitz constant independent of $\varepsilon$.

\begin{prop}\label{nolossemb}
%There exists $T_0>0$ independent of $\eps$ such that, 
Let $\varepsilon \ls r$. Then, for every $t\in [0,\min(T_\varepsilon, T_1))$, 
$M_t \cap B_i$ can be written as a Lipschitz graph over $\omega_i^\perp$, 
with Lipschitz constant independent of $\varepsilon.$
%In particular, there is no loss of embededness for $M_t$ on $[0,T_0).$
\end{prop}

\begin{proof}
We want to show that the quantity $\scal{\nu}{\omega_i}$ is bounded from below, or, equivalently, that $v:=\scal{\nu}{\omega_i}^{-1}$ is bounded from above on every ball $B_i$. We want to estimate the quantity $v^2 \phi^2$
(we drop the explicit dependence on the index $i$) using Lemma \ref{lemcgrad}. 

We choose $\Lambda$ such that the last term in \eqref{last}
is nonpositive (take for instance $\Lambda = 2NR$). We also have to control
$$v\scal{\nablas k_\varepsilon}{\omega} = \scal{\nu}{\omega}^{-1} \left( \scal{\nabla k_\varepsilon}{\omega} - \scal{\nabla k_\varepsilon}{\nu}\scal{\nu}{\omega}\right) = \scal{\nu}{\omega}^{-1} \scal{\nabla k_\varepsilon}{\omega} - \scal{\nabla k_\varepsilon}{\nu}.$$
Proposition \ref{propballs} provides immediately
\begin{align*}
 \scal{\nu}{\omega}^{-1} \scal{\nabla k_\varepsilon}{\omega} - \scal{\nabla k_\varepsilon}{\nu} & \gs \scal{\nu}{\omega}^{-1} \frac{|\nabla k_\varepsilon|}{2} - |\nabla k_\varepsilon|
\end{align*}
which is nonnegative as soon as $\scal{\omega}{\nu} \ls \frac 12$. 
% Nothing to say on the boundedness of \nabla \phi / \phi ???
{}From Lemma \ref{lemcgrad}
and the weak maximum principle (see \cite{pro84}), 
we obtain that $\Vert v^2\phi^2 \Vert_\infty (t) \ls \max (\Vert v^2 \phi^2 \Vert_\infty (0), 4 R^2)$. Thanks to Lemma \ref{lemloc}, this provides a uniform Lipschitz bound on the whole $M_t$, for $t \ls T_1$. 
%In particular, there is no loss of embeddedness before this time. 
\end{proof}

Recalling Theorem 8.1 in \cite{huisken84}, from Proposition \ref{nolossemb}
it follows that, if $T_\varepsilon < T_1$, the second fundamental form 
of $M_t$ blows up as $t \to T_\varepsilon$.
Let us show that it does not happen.

\begin{prop}\label{second}
For every $\varepsilon \ls r$, there exists $C_\varepsilon>0$ such that 
\[
\|A\|_{L^\infty(M_t)}\ls C_\varepsilon \qquad \text{for all }t\in [0,\min(T_\eps,T_1)).
\] 
\end{prop}

\begin{proof}
As in \cite{ecker91}, we are interested in the evolution of the quantity
$$
\frac{\phi^2 |A|^2 \psi(v^2)}{2}.
$$ 
Notice that
$$
|\lambda_i|^3 =|\lambda_i||\lambda_i|^2 \ls \frac{1}{2\alpha} \lambda_i^4 + \frac{\alpha}{2} \lambda_i^2.
$$
Choosing $\alpha$ such that $\frac{2N}{\alpha} \ls \frac{\gamma }{2}$, one can write
$$\left \vert -2k_\varepsilon \phi^2 \psi(v^2) \sum_i \lambda_i^3 \right \vert \ls \phi^2 \psi(v^2)\left( \frac{\gamma }{2} |A|^4 + N\alpha |A|^2\right).$$
In addition, as soon as $|A|^2 \gs 1$, one has $\scalv{A}{\nabla^2 k_\varepsilon} \ls |A|^2 |\nabla^2 k_\varepsilon|.$ One can also notice that as above, $v \scal{\nablas k_\varepsilon}{\omega} \gs 0$ as soon as $ v \gs 2.$ On the other hand, if $v \ls 2$, one has $v^3 \psi'(v^2) = \frac{\psi(v) v}{1-\gamma v^2} \ls 4 \psi(v)$ for $\gamma $ sufficiently small.

So, anyway, if $|A| \gs 1$,
$$
  \left( \frac{d}{dt} - \Deltas\right) \frac{\phi^2 |A|^2 \psi(v^2)}{2} 
\ls 2N\alpha \frac{\phi^2 |A|^2 \psi(v^2)}{2} + 4 |\nabla^2 k_\varepsilon| \frac{\phi^2 |A|^2 \psi(v^2)}{2}
 +8 \frac{\phi^2 |A|^2 \psi(v^2)}{2} |\nablas k_\varepsilon|. 
$$
Finally, we apply the maximum principle to
$$\tilde A := \exp \left[-\left( 2N\alpha+4 \Vert \nabla^2 k_\varepsilon \Vert_\infty + 8 \Vert \nabla k_\varepsilon \Vert_\infty \right) t \right] \cdot \frac{\phi^2 |A|^2 \psi(v^2)}{2} $$
which satisfies
$$ \left( \frac{d}{dt} - \Deltas \right) \tilde A \ls 0.$$
 It provides 
$$ \forall t \ls \min(T_\varepsilon, T_1),\quad \Vert \tilde A \Vert_\infty(t) \ls \Vert \tilde A \Vert_\infty(0)$$
which shows that $\frac{\phi^2 |A|^2 \psi(v^2)}{2}$ does not blow up.

Using Lemma \ref{lemloc} and choosing $\gamma $ such that $\psi(v^2)$ is bounded and remains far from zero, we know that $|A|$ does not blow up for $t \ls T_1$.
\end{proof}

\begin{cor}
There exists $T_1$, depending only on the dimension, $\Vert k \Vert_\infty$ and the radius in the ball condition for $M_0$, such that there exists a solution $M^\varepsilon_t$ of the mean curvature flow with forcing term $k_\varepsilon$ on $[0,T_1).$
\end{cor}

%\subsection{Passing to the limit as $\eps\to 0$}

The surfaces $M_t^\varepsilon$ are uniformly Lipschitz and every 
$M_t^\varepsilon \cap B_i$ can be written as the graph of some  function $u_i^\varepsilon(x,t)$. All the $u_i^\varepsilon$ are Lipchitz (in space) with a constant which depends neither on $i$ nor in $\varepsilon$. We want to show that they are also equicontinuous in time. 

\begin{prop}\label{procinq}
The functions $u_i^\varepsilon$ are Lipschitz continuous in $x$ 
and $1/2$-H\"older continuous in $t$ on $B_i\times [0,T_1)$, 
uniformly with respect to  $\eps$ and $i$.
%The motion speed of $M^\varepsilon_t$ is bounded. This bound depends only on the dimension and the Lipschitz constant of $M_t^\varepsilon.$
\end{prop}

\begin{proof}
Let $\delta$ be fixed (we drop the index $\varepsilon$ in what follows), and let $t_0 \in [0,T_1)$. Let $x_0 \in M_t$ and $i$ such that $x_0 \in B_i$. Then, $\scal{\nu(x_0)}{\omega_i} \ls C$ and $M_t$ is the graph of a function $u$ over $\omega_i^\perp.$ Then, let $x_1 = x_0 + \delta \omega_i$. Thanks to the Lipschitz condition, there is a ball $B_{1/C \delta}(x_1)$ that does not touch $M_t$. Evolving by mean curvature with forcing term $k_\varepsilon$, this ball vanishes in a positive time $T_\delta \gs \omega(\delta) 
:= \frac{\delta^2}{C^2 (2d+1)}$ (note that $T_\delta$ does not depend on $\varepsilon$). 
By comparison principle, for $t \in [t_0,t_0+\omega(\delta))$, $M_t$ does not go beyond $x_1$. That is equivalent to say that $u$ is 
$1/2$-H\"older continuous in time, with a constant independent of $\varepsilon$.
\end{proof}

We now pass to the limit as $\varepsilon$ goes to zero. 
By Proposition \ref{procinq}, 
the family $(u_i^\varepsilon)$ is equi-Lipschitz in space and equi-continuous in time 
on $B_i\times [0,T_1)$. Therefore,
by Arzel\`a--Ascoli's Theorem one can find a sequence 
$\varepsilon_n \to 0$ and continuous functions $u_i$ such that, for every $i$,  
$u_i^{\varepsilon_n} \underset{n\to \infty}{\longrightarrow} u_i$ 
locally uniformly on $B_i\times  [0,T_1)$. 
%We can extract again such that for every $i$, $u^{\varepsilon_n}_i \to u^0_i$.
 
\begin{prop}\label{provisco}
The functions $u_i$ are viscosity solutions 
of \eqref{eqg} on $B_i\times [0,T_1)$, with obstacles $U \cap B_i$
(see Appendix \ref{appvisco}). 
\end{prop}

\begin{proof}  
Thanks to Proposition \ref{propballs}, 
every $x \in B_i$ can be decomposed as $x = x' + z \omega_i$ with $z=\scal{x}{\omega_i}$. Then,
there exists functions $\psi_i^\pm$ of class $C^{1,1}$ such that 
$$U \cap B_i =\{(x',z) \in B_i \, : \, \psi_i^- (x') \ls z \ls \psi_i^+(x') \}.$$
For simplicity we shall drop the explicit dependence on the index $i$.
Since $u^\varepsilon(x,0)=u_0(x)$ for all $\varepsilon$, 
and $u^{\varepsilon_n}$ converges uniformly to $u$ as $n\to +\infty$, it is clear that $u(x,0)=u_0(x).$\\
Condition \eqref{obstacles} immediately follows from Proposition \ref{propob}.\\
We now check that $u$ is a subsolution of \eqref{eqg}. Let $(x_0,t_0)\in \Rn \times \R$ and $\varphi \in C^2$ such that $\psi^-(x_0,t_0) < u(x_0,t_0)$ and $$(u-\varphi)(x_0,t_0) = \max_{|(x,t)-(x_0,t_0)| \ls r} (u-\varphi) (x,t). $$
One can change $\varphi$ so that $(x_0,t_0)$ is a strict maximum point, and $u(x_0,t_0)=\varphi(x_0,t_0)$. 
Let $2\delta := u(x_0,t_0) - \psi^-(x_0,t_0)$. Thanks to the definition of $\gep$, for all $\varepsilon \ls \delta$, we have $\gep(x,\varphi(x,t)) \ls 0$ in a small neighborhood $V$ of $(x_0,t_0)$. Hence, for $\varepsilon$ sufficiently small $u^\varepsilon-\varphi$ attains its maximum in $V$ at $(x_\varepsilon,t_\varepsilon)$, 
with $(x_\varepsilon,t_\varepsilon)\to (x_0,t_0)$ as $\varepsilon\to 0$. Since $u^\varepsilon$ is a classical solution of 
\eqref{approxpb}, it is also a viscosity solution, therefore
$$\varphi_t - \sqrt{1+|\nabla \varphi|^2}  \div \left( \frac{\nabla \varphi}{\sqrt{1+|\nabla \varphi|^2}} \right) \ls \sqrt{1+|\nabla \varphi|^2} \,k_\varepsilon(x,\varphi) \ls 0\qquad {\rm at }(x_\varepsilon,t_\varepsilon).$$
Letting $\varepsilon\to 0$ we obtain that $u$ is a subsolution of \eqref{eqg}. A similar argument shows that $u$ is also a supersolution of \eqref{eqg},
and this concludes the proof.
\end{proof}

\noindent{\it Conclusion of the proof of Theorem \ref{main}.} 
The result in \cite[Theorem 4.1]{pet07}
(see also Section \ref{secsha}) applies, 
showing that the functions $u_i$ are of class $C^{1,1}.$
As the uniform convergence $u_i^{\varepsilon_n}$ implies the Hausdorff convergence of $M_t^{\varepsilon_n}$ to a limit $M_t$ such that 
$M_t \cap B_i = \operatorname{graph}(u_i(t))$, 
we built a $C^{1,1}$ evolution to the mean curvature motion with obstacles on the time interval $[0,T_1).$ 
Thanks to \cite[Theorem 4.8 and Corollary 4.9]{chambolle12} this evolution is also unique. This concludes the proof of Theorem \ref{main}.
\qed

\section{Proof of Theorem \ref{mainbis}}

Let $\psi_\eps^\pm$ be smooth functions such that 
$\psi_\eps^\pm\to \psi^\pm$ as $\eps\to 0$, uniformly in 
$C^{1,1}(\R^d)$, and let $N>0$ be such that
\begin{equation}\label{eqN}
N\ge \left\|\sqrt{1+|\psi_\eps^\pm|^2}
\div \left( \frac{\psi_\eps^\pm}{\sqrt{1+|\psi_\eps^\pm|^2}} \right)\right\|_{L^\infty(\R^d)} \qquad \text{for all }\eps>0.
\end{equation}
We proceed as in Section \ref{geometric} and we approximate 
\eqref{eqg}, \eqref{eqcongr} with the forced mean curvature equation
\begin{equation}\label{approxpb}
u_t = \sqrt{1+\vnau^2}\left[\div \left( \frac{\nau}{\sqrt{1+\vnau^2}} \right) 
+k_\eps(x,u)\right],
\end{equation}
%\subsection{A mean curvature flow with a forcing term}
%\label{approximatepb}
%The main aim of this section is to obtain more regularity on the viscosity solution if the initial data and the obstacles are more regular. More precisely, we assume that $u_0$ and $\psi^\pm$ are $C^{1,1}$.
%To get more regularity than the one expected with traditional viscosity background, we proceed as in Section \ref{geometric}. let us introduce a new problem. 
%We define $g$ to be a forcing term
where
$$k_\varepsilon (x,u) = 2N \left( \chi \left( 
\frac{\psi_\eps^-(x)-u}{\varepsilon} \right) 
- \chi \left( \frac{u-\psi_\eps^+(x)}{\varepsilon} \right) \right),$$
and $\chi$ is a smooth increasing function such that $\chi(s)\equiv 0$ for all $s\in (-\infty,0]$, 
and $\chi(s)\equiv 1$ for all $s\in [1,\infty)$. In particular $\partial_u k_\eps(x,u)\le 0$ for all $(x,u)$.

Notice that $k_\eps\to g$ as $\eps\to 0$, with
$$k(x,u)= \left\{ \begin{matrix} 2N \text{ if } u < \psi^-(x) 
\\ 
-2N \text{ if } u>\psi^+(x) 
\\ 
0 \text{ elsewhere}\end{matrix} \right..
$$
Notice also that 
%for all $\delta \gs0$, 
%$$k_\varepsilon(x,\psi^- -\delta) = N \chi \left( \frac \delta \varepsilon \right),$$
%$$k_\varepsilon(x,\psi^+ + \delta) = -N \chi \left( \frac \delta \varepsilon \right).$$
%In particular,

\begin{equation}\label{diffg}
\begin{aligned}
 \frac{\partial k_\eps}{\partial x_k} (x,u) + \frac{\partial k_\eps}{\partial u} (x,u) \frac{\partial \psi_\eps^-}{\partial x_k} = 0 \quad \text{ if } u < \psi_\eps^+  \\
 \frac{\partial k_\eps}{\partial x_k} (x,u) + \frac{\partial k_\eps}{\partial u} (x,u) \frac{\partial \psi_\eps^+}{\partial x_k} = 0 \quad \text{ if } u >\psi_\eps^-.  
\end{aligned}
\end{equation}

We denote by $u_\varepsilon$ the solution of the approximate problem
\eqref{approxpb}, which exists and is smooth for short times. 

\begin{prop}\label{spatial}
The solution $u_\eps$ is defined for $t\in [0,+\infty)$,
and satisfies the estimates 
\begin{eqnarray}\label{estone}
\| u_\eps(\cdot,t)\|_{W^{1,\infty}(\R^d)} &\le& C \qquad\quad\ \text{for all }t\in [0,+\infty)  
\\\label{esttwo}
\| u_\eps(\cdot,t)\|_{W^{2,\infty}(\R^d)} &\le& C(T) \qquad \text{for all }t\in [0,T]. \end{eqnarray}
\end{prop}

\begin{proof}
Estimate \eqref{estone} follows from
Proposition \ref{nolossemb}, choosing $B_i=\R^{d+1}$, $\omega_i=e_{d+1}$ and $\phi \equiv 1$. 
Estimate \eqref{esttwo} follows from \eqref{estone} and Proposition \ref{second}. 
\end{proof}

%Moreover, for all $\varepsilon$, $u_\varepsilon$ share the same time Hölder constant in time (provided by the Lipschitz constant in space). This enables to pass to the limit in every compact set.

%If we let $\varepsilon$ go to $0$, thanks to Theorem \ref{gradest}, $(u_\varepsilon)$ is equi Lipschitz in space and time. Then, up to a subsequence, it converges uniformly on compact sets to some function $u$ which has the same Lipschitz bounds. We will see at the end of this section, that the function $u$ is the viscosity solution to \eqref{pb}.

%With the specific graph framework and using the bound on the mean curvature of $u_0$ and $\psi^pm$, one can prove that the solution is in fact Lipschitz in time.

%\subsection{The solution is Lipschitz in time}
%We work on the approximate solutions $u_\varepsilon$ and show that it has a Lipschitz bound which does not depend on $\varepsilon$.

%\paragraph{Time derivative estimates.}
In what follows, we use intrinsic derivatives on the graph 
$M_t:=\{(x,u_\varepsilon(x,t))\}$, which will be denoted as above by an exponent $S$. 
The metric on $M_t$ is
$$g_{ij} = \delta_{ij} + \partial_i u_\varepsilon \partial_j u$$
with inverse
$$g^{ij} = \delta_{ij} - \frac{\partial_iu_\varepsilon \partial_ju_\varepsilon}{1+\vnaue^2}.$$
The tangential gradient of a function $f$ defined on $M_t$ is given by
$$(\nablas f)^i = g^{ij} \partial_j f = \partial_i f - \frac{\partial_i u_\varepsilon \partial_j u_\varepsilon}{1+\vnaue^2} \partial_j f\,,
$$
so that
\begin{equation}\scal{\nablas f}{\nabla u_\varepsilon} = \scal{\nabla f}{\nabla u_\varepsilon} - \frac{\vnaue^2}{1+\vnaue^2} \scal{\nabla f}{\nabla u_\varepsilon} = \frac{1}{1+\vnaue^2} \scal{\nabla f}{\nabla u_\varepsilon}, \label{nablascal}\end{equation}
and
\begin{equation}
\begin{aligned}
 | \nablas f |^2 &=  \left( f_i - (u_\varepsilon)_i\sum_j \frac{ (u_\varepsilon)_j f_j}{1+\vnaue^2} \right)^2 \\
&= |\nabla f|^2 +  (u_\varepsilon)_i^2 \left( \frac{\scal{\naue}{\nabla f}}{1+\vnaue^2} \right)^2 - 2  \frac{(u_\varepsilon)_i (u_\varepsilon)_j f_i f_j}{1+\vnaue^2} \\
&= |\nabla f|^2 + \frac{\vnaue^2}{1+\vnaue^2} \frac{\scal{\naue}{\nabla f}^2}{1+\vnaue^2} - 2\frac{\scal{\naue}{\nabla f}^2}{1+\vnaue^2} \\
&= |\nabla f|^2  - \frac{\scal{\naue}{\nabla f}^2}{1+\vnaue^2} - \frac{\scal{\naue}{\nabla f}^2}{(1+\vnaue^2)^2}.
\end{aligned}
\label{normgrads}
\end{equation}
In addition, the Laplace-Beltrami operator applied to $f$ is
$$\Deltas f = g^{ij} f_{ij} = \Delta f - \frac{\partial_i u_\varepsilon \partial_j u_\varepsilon}{1+\vnaue^2} f_{ij} = \Delta f - \frac{ \scal{\nabla u_\varepsilon \nabla^2f}{\nabla u_\varepsilon}}{1+\vnaue^2}.$$
%One can now get the Lipschitz bound in time for $u_\varepsilon.$
\begin{prop} The quantity $\Vert (u_\varepsilon)_t^2 \Vert_\infty(t)$ is nonincreasing in time. In particular,
$$
\|(u_\varepsilon)_t(\cdot,t)\|_{L^\infty(\R^d)}\le
\left\|\sqrt{1+|\nau_0|^2}\div \left( \frac{\nau_0}{\sqrt{1+|\nau_0|^2}} \right)
\right\|_{L^\infty(\R^d)}.
$$
\label{timebound}
\end{prop}

\begin{proof}
We compute
$$ \frac{d}{d t} \frac{(u_\varepsilon)_t^2}{2}  = (u_\varepsilon)_t \left[ \sqrt{1+|\nabla u_\varepsilon|^2}\left(\div\left( \frac{\nabla u_\varepsilon}{\sqrt{1+\vnaue^2}}\right) + k_\varepsilon(x,u_\varepsilon\right) \right]_t. $$
Expanding this expression, we get
\begin{multline*} \frac{d}{d t} \frac{(u_\varepsilon)_t^2}{2} = (u_\varepsilon)_t\left[ \frac{\nabla (u_\varepsilon)_t \cdot \nabla u_\varepsilon}{\sqrt{1+\vnaue^2}} \left( \div \left( \frac{\naue}{\sqrt{1+\vnaue^2}} \right) + k_\varepsilon \right) \right. \\ \left. + \sqrt{1+\vnaue^2}\left( \div\left( \frac{(\naue)_t}{\sqrt{1+\vnaue^2}} - \frac{((\naue)_t \cdot \naue) \naue}{(1+\vnaue^2)^{3/2}} \right) + (u_\varepsilon)_t \partial_u k_\varepsilon \right) \right].  
\end{multline*}

Let us compute more explicitly the three terms of the expression above.
\begin{align*}(u_\varepsilon)_t & \frac{(\naue)_t \cdot \naue}{\sqrt{1+\vnaue^2}}\left( \div \left( \frac{\naue}{\sqrt{1+\vnaue^2}} \right) + k_\varepsilon \right)    \\
&= \frac{\nabla(\frac{(u_\varepsilon)_t^2}{2}) \cdot \nabla u_\varepsilon}{\sqrt{1+\vnaue^2}} \left( \frac{\Delta u}{\sqrt{1+\vnaue^2}} - \frac{(u_\varepsilon)_i \scal{\naue}{(\naue)_i}}{(1+\vnaue^2)^{3/2}} + k_\varepsilon \right) \\
&= \nabla(\frac{(u_\varepsilon)_t^2}{2}) \cdot \nabla u_\varepsilon \left( \frac{\Delta u_\varepsilon}{1+\vnaue^2} -\frac{\naue \cdot \nabla(\frac{\vnaue^2}{2})}{(1+\vnaue^2)^{2}} + k_\varepsilon \right), \\
\end{align*}
\begin{align*}
 (u_\varepsilon)_t \div\left(\frac{\nabla (u_\varepsilon)_t}{\sqrt{1+\vnaue^2}}\right) &= (u_\varepsilon)_t \partial_i \left( \frac{(u_\varepsilon)_{ti}}{\sqrt{1+\vnaue^2}} \right) \\ &=  \frac{(u_\varepsilon)_t (u_\varepsilon)_{tii}}{\sqrt{1+\vnaue^2}}-\frac{1}{(1+\vnaue^2)^{3/2}} (u_\varepsilon)_t (u_\varepsilon)_{ti} \nabla u_\varepsilon \cdot (\naue)_i \\
&=\frac{(u_\varepsilon)_t \Delta (u_\varepsilon)_t}{\sqrt{1+\naue^2}} - \frac{1}{(1+\vnaue^2)^{3/2}}  (u_\varepsilon)_t (u_\varepsilon)_{ti} \partial_i(\frac{\vnaue^2}{2})\\
&= \frac{(u_\varepsilon)_t \Delta (u_\varepsilon)_t}{\sqrt{1+\naue^2}} - \frac{1}{(1+\vnaue^2)^{3/2}} \nabla\left( \frac{(u_\varepsilon)_t^2}{2} \right) \cdot \nabla \left(\frac{\vnaue^2}{2} \right),
\end{align*}
and
\begin{align*}
(u_\varepsilon)_t& \div\left( \frac{((\naue)_t \cdot \naue)\naue}{(1+\vnaue^2)^{3/2}} \right) \\
&=  \Delta u_\varepsilon \frac{\scal{\nabla u_\varepsilon}{(u_\varepsilon)_t \nabla (u_\varepsilon)_t}}{(1+\vnaue^2)^{3/2}} +  \frac{(u_\varepsilon)_t (u_\varepsilon)_{tij} (u_\varepsilon)_j (u_\varepsilon)_i}{(1+\vnaue^2)^{3/2}} + \frac{\scal{(u_\varepsilon)_i \nabla (u_\varepsilon)_i}{(u_\varepsilon)_t \nabla (u_\varepsilon)_t}}{(1+\vnaue^2)^{3/2}} 
\\
&\ \ - 3 (u_\varepsilon)_i\frac{\scal{(u_\varepsilon)_t \nabla (u_\varepsilon)_t}{\nabla u_\varepsilon} \scal{\nabla (u_\varepsilon)_i}{\nabla u_\varepsilon}}{(1+\vnaue^2)^{5/2}} \\
&= \Delta u_\varepsilon \frac{\scal{\nabla u_\varepsilon }{\nabla (\frac{(u_\varepsilon)_t^2}{2})}}{(1+\vnaue^2)^{3/2}} +  \frac{(u_\varepsilon)_t (u_\varepsilon)_{tij} (u_\varepsilon)_j (u_\varepsilon)_i}{(1+\vnaue^2)^{3/2}} + \frac{\scal{\nabla (\frac{\vnaue^2}{2})}{\nabla (\frac{(u_\varepsilon)_t^2}{2})}}{(1+\vnaue^2)^{3/2}} 
\\
&\ \ -3\frac{ \scal{\nabla (\frac{(u_\varepsilon)_t^2}{2})}{\nabla u_\varepsilon} \scal{\nabla (\frac{\vnaue^2}{2})}{\naue}}{(1+\vnaue^2)^{5/2}}. 
\end{align*}
Notice that
\begin{align*}
 \Deltas \frac{(u_\varepsilon)_t^2}{2} &= \Delta \frac{(u_\varepsilon)_t^2}{2} 
- \frac{\scal{\naue}{\nabla^2\frac{(u_\varepsilon)_t^2}{2} \naue}}{1+\vnaue^2} \\
&= (u_\varepsilon)_t \Delta (u_\varepsilon)_t + |(\naue)_t|^2 -  \frac{(u_\varepsilon)_i (u_\varepsilon)_j (u_\varepsilon)_t (u_\varepsilon)_{tij} + (u_\varepsilon)_i (u_\varepsilon)_j (u_\varepsilon)_{ti} (u_\varepsilon)_{tj}}{1+\vnaue^2}.
\end{align*}
We then get

\begin{multline*}
\frac{d}{d t} \frac{(u_\varepsilon)_t^2}{2} =  \frac{\scal{\nabla (\frac{(u_\varepsilon)_t^2}{2})}{\nabla u_\varepsilon}}{\sqrt{1+\vnaue^2}} k_\varepsilon
+ \Deltas \left( \frac{(u_\varepsilon)_t^2}{2} \right)   - 2\frac{\scal{\nabla\left( \frac{(u_\varepsilon)_t^2}{2} \right)}{\nabla \left(\frac{\vnaue^2}{2}\right)}}{1+\vnaue^2} \\
+2 \frac{\scal{\nabla (\frac{(u_\varepsilon)_t^2}{2})}{\nabla u_\varepsilon} \scal{\nabla (\frac{\vnaue^2}{2})}{\naue}}{(1+\vnaue^2)^{2}} 
+\frac{\scal{\naue}{(\naue)_t}^2}{1+\vnaue^2}-|(\naue)_t|^2 + (u_\varepsilon)_t^2 \partial_u k_\varepsilon.
\end{multline*}
%We eventually note that $k_\varepsilon(x,\psi^\pm (x) \pm \varepsilon)$ remains constant, 
Note that the last term is nonpositive by definition of $k_\varepsilon$.

In order to apply Lemma \ref{lemh}, we have to show the inequality
$$-\frac{\scal{\naue}{(\naue)_t}^2}{1+\vnaue^2}+|(\naue)_t|^2 \gs 0.$$
It is enough to note that, since the solution exists for all times and 
it is smooth, the term $\nabla (\frac{\vnaue^2}{2})$ is bounded on each $[0,T]$ (the bound depends on $T$ and $\varepsilon$ but is enough to apply the lemma). In addition, 
%we assumed $\partial_u g \leqslant 0$ and can note that 
every factor containing $\nabla ((u_\varepsilon)_t^2/2)$ 
also contains $\nabla u_\varepsilon$, 
hence the assumptions of Lemma \ref{lemh} are satisfied for every $T>0$,
and this concludes the proof.
\end{proof}

{}From Propositions \ref{spatial} 
and \ref{timebound}, we deduce the following result.

\begin{prop}
If $u_0$ is $C$-Lipschitz in space for some $C>0$, 
and has bounded mean curvature, then the solution $u_\varepsilon$ of the approximate problem \eqref{approxpb} is $C$-Lipschitz in space and Lipschitz in time with constant
$$ \left\|\sqrt{1+|\nau_0|^2}\div \left( \frac{\nau_0}{\sqrt{1+|\nau_0|^2}} \right)
\right\|_{L^\infty(\R^d)}.$$
Moreover, the following inequalities hold
\begin{equation}\psi_\varepsilon^-(x)-\varepsilon 
\le u_\varepsilon(x,t) \le \psi_\varepsilon^+(x)+\varepsilon. \label{obstapprox}\end{equation}
\label{gradest}
\end{prop}

\begin{proof}
The Lipschitz bounds of the solution are clear (it is Proposition \ref{spatial} and \ref{timebound}). 

In order to prove the second assertion, 
let us notice that by \eqref{eqN} and the definition of $k_\eps$, 
we have 
$$
k_\eps(x,\psi_\varepsilon^- - \varepsilon) = 2N \gs 
\left\|\sqrt{1+|\psi_\eps^-|^2}
\div \left( \frac{\psi_\eps^-}{\sqrt{1+|\psi_\eps^-|^2}} 
\right)\right\|_{L^\infty(\R^d)}\,,
$$ 
so that $\psi_\varepsilon^- - \varepsilon$ is a subsolution of \eqref{approxpb}. By the parabolic comparison principle (as in Proposition \ref{propob}), we deduce that
$$\psi_\varepsilon^- - \varepsilon \ls u_\varepsilon.$$ The same arguments shows the other inequality in \eqref{obstapprox}.
\end{proof}

\smallskip

\noindent{\it Conclusion of the proof of Theorem \ref{mainbis}.}
Since the solutions $u_\varepsilon$ are equi-Lipschitz in space and time, 
they converge uniformly, as $\varepsilon \to 0$, to a limit function $u$
which is also Lipschitz continuous on $\R^d\times [0,+\infty)$.\\ 
Equation \eqref{obstapprox} yields
$$\psi^- \ls u \ls \psi^+,$$
and Proposition \ref{provisco} gives that $u$ is 
a viscosity solution of \eqref{pb}.

Concerning the regularity of $u$, we proved that $(u_\varepsilon)_t$ and $\nabla u_\varepsilon$ are bounded on $[0,T],$ for any $T$ in the approximate problem. This gives a bound on the mean curvature of the approximate solution. This bound does not depend on $\varepsilon$ and remains true for the viscosity solution. As a result, the exact solution has bounded mean curvature and bounded gradient, which shows that $\Delta u$ is $L^\infty$ and, by elliptic regularity theory, $u$ is also in $W^{2,p}$ for any $p>1$, and so $C^{1,\alpha}$ for every $\alpha < 1$ (see \cite{lunardi95} for details).

By Theorem \ref{thmsha} below,
we can also directly apply to the solution $u$ a regularity result by Petrosyan and Shahgholian in \cite{sha08,pet07}. 
It follows that $u$ is in fact of class $\mathcal C^{1,1}$, and this concludes the proof of Theorem \ref{mainbis}.

\qed

\section{Proof of Theorem \ref{maintris}}
%In what follows, we assume that there is no forcing term $k$.
%This section is dedicated to prove the following result.

We compute the evolution of the area of the graph of $u$: 
\begin{equation}\label{graph}
\frac{d}{d t} \int_Q \sqrt{1+\vnau^2} = \int_Q \frac{\scal{\nabla u_t}{\nabla u}}{\sqrt{1+\vnau^2}}  
= -\int_Q u_t \div \left( \frac{\nabla u}{\sqrt{1+\vnau^2}} \right). 
\end{equation}
Notice that, for almost every $t>0$, 
$u_t(t,x) =0$ almost everywhere on the contact set. Indeed, for almost every $t$, $u_t$ exists for almost every $x\in Q$. 
If $u(x,t)=\psi^\pm (x)$, then $u-\psi^\pm$ reaches an extremum in $(x,t)$, which gives, $u_t(x,t)=0.$ In particular, from 
\eqref{graph} we get
$$
\frac{d}{d t} \int_Q \sqrt{1+\vnau^2} =-\int_Q u_t \left( \frac{u_t}{\sqrt{1+\vnau^2}} \right).
$$
Integrating this equality in time, we obtain
$$
\left. 
\int_Q \sqrt{1+\vnau^2} \right \vert_0^T = \int_0^T \int_Q -\frac{u_t^2}{\sqrt{1+\vnau^2}} . 
$$
%\begin{equation}
%\begin{aligned}\int_0^T \int_{\R^n}   u_t \frac{\scal{\nabla \phi}{\nabla u}}{\sqrt{1+\vnau^2}} & \ls \int_0^T\int_{\supp \phi} |u_t| |\nabla \phi| \\
%& \ls \left( \int_0^T\int_{\supp \phi} \phi u_t^2 \int_0^T\int_{\supp \phi} \frac{|\nabla \phi|^2}{\phi}\right)^{1/2} \\
%& \ls C\sqrt{T} \sqrt{\int_0^T \int_{\supp \phi} \phi u_t^2}
%\end{aligned} \label{gradphi}
%\end{equation}
%where $C$ bounds $\left(\int_{\supp \phi} \frac{|\nabla \phi|^2}{\phi}\right)^{1/2}$.
%In addition, one can prove that 
%So,
%\begin{equation}\left. \int_Q \sqrt{1+\vnau^2} \phi \right \vert_0^T = \int_0^T \int_Q -\frac{\phi u_t^2}{\sqrt{1+\vnau^2}} - u_t \frac{\scal{\nabla \phi}{\nabla u}}{\sqrt{1+\vnau^2}}. \label{controlH1}
%\end{equation}
% Since $|\nabla u|\le C$ by Proposition \ref{prograd}, we get
% \begin{equation} 
% \int_0^T \int_Q -u_t^2 \ls  \int_0^T \int_Q -\frac{u_t^2}{\sqrt{1+\vnau^2}} \ls \frac 1C \int_0^T \int_Q - u_t^2
% \label{controlH12} 
% \end{equation}
% and the left term in \eqref{controlH1} is bounded uniformly (say, by $L$) in $T$ and $\phi$ provided $\Vert \phi \Vert_\infty \ls 1$.
% Thanks to \eqref{gradphi} and \eqref{controlH12}, we have
% $$ L \gs  \int_0^T \int_Q \frac{\phi u_t^2}{\sqrt{1+\vnau^2}} + u_t \frac{\scal{\nabla \phi}{\nabla u}}{\sqrt{1+\vnau^2}} \gs \frac 1C \int_0^T \int_Q \phi u_t^2 - C \sqrt{T} \sqrt{\int_0^T \int_Q \phi u_t^2}$$
which shows that 
$$ \int_0^T \int_{Q} u_t^2$$ is uniformly bounded in $T$. As a result $u_t \in L^2(\R^+,Q)$ so $u$ is in $H^1(Q,B_R).$

Since $\Vert u_t \Vert_{L^2(Q)}$ is $L^2(\R^+)$, there exists a sequence $t_n \to \infty$ such that $$\Vert u_t \Vert_{L^2(Q)} (t_n) \underset{n\to \infty}{\longrightarrow} 0.$$ 
In addition, $u(t_n)$ is equi Lipschitz and converges uniformly on compact sets to some $u_\infty$ which therefore satisfies in the viscosity sense
$$\sqrt{1+\vnau^2}\div\left(\frac{\nabla u}{\sqrt{1+\nabla u^2}} \right) = 0$$ with obstacles $\psi^\pm$ (see Appendix \ref{appvisco}).

\qed 

\begin{rem}\rm
By \cite{ilm98}, $u_{\min}$ is analytic out 
of the (closed) contact set $\{u_{\min}=\psi^\pm\}$.
\end{rem}

\appendix\label{appvisco}

\section{Viscosity solutions with obstacles}
\subsection{Definition of viscosity solution}
Given an open subset $B$ of $\R^d$,
let $u_0$, $\psi^+$ and $\psi^-$ be three Lipschitz functions $B \to \R$
%with $B$  and 
%and 
%$\Vert \psi^+ \Vert_\infty$, $\Vert \psi^- \Vert_\infty$, $\Vert u_0 \Vert_\infty$, 
%$\Vert \nabla \psi^+ \Vert_\infty$, $\Vert \nabla \psi^- \Vert_\infty$ 
%and $\Vert \nabla u_0 \Vert_\infty$ bounded by $M$ 
such that $$\psi^-(x,0) \ls u_0(x) \ls \psi^+(x,0).$$ 
We are interested in the viscosity solutions of the equation
\begin{equation} u_t = \sqrt{1+|\nabla u|^2} \div \left( \frac{\nabla u}{\sqrt{1+|\nabla u|^2}} \right) , \qquad u(x,0)=u_0(x), \label{pb}
\end{equation}
with the constraint
\begin{equation}\label{eqpsi}
\psi^-(x) \ls u(x,t) \ls \psi^+(x).
\end{equation}

%\smallskip

%Following \cite{c92user} (see also \cite{mercier}), we define.

\begin{defi}[see \cite{c92user,mercier}] We say that a function $u : B\times [0,T) \to \R$ is a 
{\em viscosity subsolution} of \eqref{pb} if $u$ satisfies the following conditions:
\begin{itemize}
\item $u$ is upper semicontinuous;
\item $u(x,0) \ls u_0(x)$;
 \item \begin{equation}
 \psi^-(x) \ls u(x,t) \ls \psi^+(x); \label{obstacles}
\end{equation}
\item for any $(x_0,t_0) \in \Rn \times \R^+$ and $\varphi \in C^2$ such that $u-\varphi$ has a maximum at $(x_0,t_0)$ and $u(x_0,t_0) > \psi^-(x_0)$,
\begin{equation}
u_t \ls \sqrt{1+|\nabla u|^2} \div \left( \frac{\nabla u}{\sqrt{1+|\nabla u|^2}} \right). 
\label{subsol}
\end{equation}
\end{itemize}
%\end{defi}
%\begin{defi}
Similarly, $u$ is a {\em viscosity supersolution} of \eqref{pb} if: 
\begin{itemize}
\item $u$ is lower semicontinuous;
\item $u(x,0) \gs u_0(x)$;
\item \eqref{obstacles} holds;
\item for any $(x_0,t_0) \in \Rn \times \R^+$ and $\varphi \in C^2$ such that $u-\varphi$ has a minimum at $(x_0,t_0)$ and $u(x_0,t_0) < \psi^+(x_0)$,
\begin{equation*}
 u_t \gs \sqrt{1+|\nabla u|^2} \div \left( \frac{\nabla u}{\sqrt{1+|\nabla u|^2}} \right). 
%\label{subsol}
\end{equation*}
\end{itemize}
%\end{defi}
%\begin{defi}
We say that $u$ is 
a {\em viscosity solution} of \eqref{pb} if it is both a super and a subsolution.
\end{defi}
\smallskip

%We are now proving long time existence and uniqueness of a viscosity solution of equation \eqref{pb} with obstacles. 

%%%%%%%%%%%%%%%%%%%%%%%%%%%%%%%%%%%%%%%%%%%%%%%%%%%%%%%%%%%%%%%%%%%%%%%%%%%%

\subsection{Comparison principle}

In order to prove uniqueness of continous viscositysolutions of \eqref{pb}, 
we shall prove a comparison principle 
between solutions following \cite[Theorem 4]{giga90} (see also \cite{chen91}).
 
\begin{prop}
 If $u$ is a viscosity subsolution of \eqref{pb} on $[0,T)$, $v$ is a viscosity supersolution, if $\psi^\pm$ are Lipschitz in space and if $u(x,0) \ls v(x,0)$, then $u(x,t)\ls v(x,t)$ for all $(x,t)\in \mathbb R^n \times [0,T).$
\label{compple}
\end{prop}

\begin{proof}
We will check that the proof of  \cite[Theorem 2.1]{giga90}
can be extended to the obstacle case. 
Notice first that the assumptions $(A.1)-(A.3)$ of \cite[Theorem 2.1]{giga90} 
are satisfied also in our case. 
%$U = \Rn \times \Rn \times [0,T)$ so the parabolic boundary is only $\partial_p U= \Rn\times \Rn \times \{0\}$). 
Indeed, $(A.1)$ comes directly from the Lipschitz bound on $\psi^\pm$ and the constraint $\psi^- \ls u,v\ls \psi^+$ whereas $(A.2)$ and $(A.3)$ result from the assumed time zero comparison.

Let us show that \cite[Proposition 2.3]{giga90} also holds. 
Indeed, up to Equation (2.9) nothing chenges. 
To continue the proof, using the same notation of \cite[Proposition 2.3]{giga90}, we have to check that if
$$ \sup_{V} (w-\Psi) >0,$$ 
then the supremum is reached in the complementary of the contact set $\{u = \psi^-\} \cup \{v = \psi^+\}$. 

Indeed, notice that if $u(x,t)=\psi^-(x)$, then, for all $x,y,t,s$,
\begin{align*}
 u(x,t)-v(y,s) &= \psi^-(x)-v(y,s) \ls \psi^-(y) + L (|x-y|) - v(y,s) \ls  L (|x-y|) 
\end{align*}
since $v \gs \psi^-.$
Hence, if $u(x,t)=\psi^-(x)$, with $K'>L$, we must have $w-\Psi \ls 0$, so the supremum of $w-\Psi$ is attained in the complementary of $\{u=\psi^-\}$. One can
show similarly that the supremum is reached in the complementary of  
$\{v=\psi^+\}$. Hence Proposition 2.3 of \cite{giga90} holds.

{}From Proposition 2.4 to Lemma 2.7 of \cite{giga90}, 
every result holds without changes.

Concerning the proof of Theorem 2.1 of \cite{giga90}, the first assumption is
$$\alpha = \limsup_{\theta \to 0}\{w(t,x,y), \, \mid \, |x-y|\ls \theta\} >0.$$
Then, Proposition 2.4 gives constants $\delta_0$ and $\gamma_0$ such that for all $\delta \ls \delta_0,$ $\gamma \ls \gamma_0$ and $\varepsilon >0$,
there holds
$$\Phi(\hat x, \hat y , \hat t) := \sup_{\Rn \times \Rn \times [0,T)} \Phi(x,y,t) > \frac \alpha 2$$
with
$$\Phi(t,x,y) = u(x,t)-v(y,t) - \frac{|x-y|^4}{4\varepsilon} - \delta(|x|^2 + |y|^2) - \frac{\gamma}{T-t}$$

To conclude the proof, we only have to show that the maximum of $\Phi$ is once again attained on the complementary of $\{ u = \psi^-\} \cup \{v = \psi^+\}.$ 
In the same way as for Proposition 2.3, if $u(x,t)=\psi^-(x)$, we can write
\begin{align*}
 \Phi(t,x,y) &= u(x,t)-v(y,t) - \frac{|x-y|^4}{4\varepsilon} - \delta(|x|^2 + |y|^2) - \frac{\gamma}{T-t} \\
 &\ls \psi^-(y) + L|x-y| - v(y,t) \ls L|x-y|.
\end{align*}

Thanks to Proposition 2.5, $|\hat x - \hat y| \underset{\varepsilon \to 0}{\longrightarrow} 0$. So, with $\varepsilon$ sufficiently small (one can reduce the 
quantity $\varepsilon_0$ given by Proposition 2.6), $\Phi$ has its maximum out of $\{u=\psi^-\}$ (and similarly out of $\{v=\psi^+\}$), which enables the application of Lemma $2.7$ and gives a contradiction as in \cite{giga90}.
\end{proof}

\subsection{Existence}
\label{vexistence}
In this subsection, we prove the following result:
\begin{prop}
There exists a continuous viscosity solution to \eqref{pb}.
\end{prop}

We follow \cite{c92user} to build a solution by means of the Perron's method. Let us state an obvious but useful proposition and 
a key lemma for applying Perron's method.

\begin{prop}
Let $u$ be a subsolution of the mean curvature motion for graphs 
(without obstacles) which satisfies $u \ls u^+$. 
Then, $u_{ob}:=u\vee u^-$ is a subsolution of \eqref{pb} with obtacles (the same happends for $v$ supersolution and $v_{ob} = v \wedge u^+$). \label{wedgevee}
\end{prop}

In the sequel, we shall denote by $u^\ast$ (resp. $u_\ast$) the upper
(resp. lower) semicontinuous envelope of a function $u$. 

\begin{lem}
Let $\mathcal F$ be a family of subsolutions of \eqref{pb}. We define 
$$U(x,t) = \sup\{u(x,t) \, \mid \, u \in \mathcal F\}.$$
Then, $U^\ast$ is a subsolution of \eqref{pb}.
\label{supsol}
\end{lem}
The proof of the proposition and the lemma can be found in \cite{c92user}, Lemma 4.2 (with obvious changes due to the parabolic situation and obstacles).

\paragraph{Construction of barriers}
In the sequel, to claim that the initial condition is taken by the viscosity solution, we need to build barriers to sandwich the solution. More precisely, we want to build a subsolution $w^-$ such that $(w^-)^\ast(x,0) = u_0(x)$ and a supersolution $w^+$ such that $(w^+)_\ast (x,0)=u_0(x).$ 
To show this claim, let us begin by a simple fact.

Let \begin{equation}g_{\alpha,b}^a(x)=-\sum \alpha_i \frac{(x-a)_i^2}{\sqrt{1+(x-a)_i^2}} + b\label{defbar} \end{equation} for some $(a,b) \in \Rn \times \R$ and $\alpha_i \gs 0$ such that $g(x) \ls u_0(x)$. Note in particular that 
\begin{equation} g_{\alpha,b}^a(x) \gs - \sum \alpha_i (x-a)_i^2 +b \quad \text{and} \quad H(g_{\alpha,b}^a)\gs H(g_{\alpha,b}^a)\vert_{t=0} = -2\sum \alpha_i.\label{propbar}\end{equation}

Then, it is easy to show (using Proposition \ref{wedgevee}) that the function $$v(x,t) = \left( g_{\alpha,b}^a(x) +\left( 2\sum_{i=1}^n \alpha_i + 3M \right) t  \right) \vee \psi^-$$ is a subsolution of \eqref{pb}. Indeed, the curvature of $g_{\alpha,b}^a$ is smaller than $2\sum \alpha_i$ and its gradient is bounded by $2$ (so $\sqrt{1+\vert \nabla g \vert^2} \ls 3$).

Thanks to Lemma \ref{supsol}, the function
$$w^-(x,t)= \left(\sup_{\substack{(\alpha_i), c \\ g_{\alpha,b}^c \ls u_0}} \left( g_{\alpha,b}^a(x) -2\sum_{i=1}^n \alpha_i t - 3 Mt  \right) \vee \psi^-\right)^\ast $$ is a subsolution of \eqref{pb} (with obstacles). \\
It remains to show that $(w^-)^\ast (x,0) = u_0(x).$ To see this, notice that since $u_0$ is Lipschitz and $u_0 \gs \psi^-$, $u_0(x) = w^-(x,0)$, yielding $u_0(x) \ls (w^-)^\ast (x,0)$. But for all $t \gs 0$, $v(x,t) \ls u_0(x)$ so $w^- (x,t) \ls u_0(x).$ By continuity of $u_0$, $(w^-)^\ast (x,t) \ls u_0(x)$, which shows that $(w^-)^\ast (x,0)=u_0(x)$, and $w^-$ is a low barrier for solutions of \eqref{pb}. 

We build $w^+$ in the same way.

\paragraph{Perron's method}
We use the classical Perron's method to build a solution of \eqref{pb} on $[0,T)$ for every $t>0.$
Let us define
$$W(x,t)=\sup\{u(x), \, \mid \, u \text{ is a subsolution of \eqref{pb} on }[0,T) \}.$$
Since $\psi^-$ is a subsolution, this set in non empty and $W$ is well defined. Every subsolution is less that $\psi^+$, so is $W$.

Thanks to Lemma \ref{supsol}, $W^\ast$ is a subsolution of \eqref{pb} regardless the initial conditions. Applying the comparison principle (Proposition \ref{compple}) to every subsolution $u$ and $w^+$ gives 
$$\forall x,t, \; W(x,t) \ls w^+(x,t).$$
Considering the upper-semi-continuous envelopes, we get
$$\forall x,t, \; W^\ast(x,t) \ls (w^+)^\ast(x,t)$$
which immediately yields to
$$W^\ast(x,0) = u_0(x).$$
Then, $W^\ast$ is a subsolution (with initial conditions), hence $W^\ast = W$ which shows the upper semi-continuity of $W$.

We want to prove that $W$ is actually a solution of \eqref{pb}. In this order, let us prove the following
\begin{lem}
 Let $u$ be a subsolution of \eqref{pb}. If $u_\ast$ fails to be a supersolution (regardless initial conditions) at some point $(\hat x,\hat t)$ then there exists a subsolution $u_\kappa$ (regardless initial conditions) satisfying $u_\kappa \gs u$ and $\sup u_\kappa - u >0$ and such that $u(x,t)=u_\kappa(x,t)$ for $|x-\hat x|,|t-\hat t| \ls \kappa$.
\end{lem}
\begin{proof}
 Let us assume that $u_\ast$ fails to be a supersolution at $(0,1)$. Then there exists $(a,p,X) \in \mathcal J^{2,-} u_\ast (0,1)$ with
$$a + F(p,X) + k(0) \sqrt{1+p^2} < 0.$$
Let us then define
$$u_{\delta,\gamma} (x,t)=u_\ast(0,1) + \delta + \scal{p}{x} + a(t-1) + \frac{1}{2} \scal{Xx}{x} - \gamma (|x|^2+t-1).$$
Thanks to the continuity of $F$ and $k$, $u_{\delta,\gamma}$ is a classical subsolution on $B_r(0,1)$ of $u_t+F(Du,D^2u)+k(x)\sqrt{1+\vnau^2} =0$ for $\delta,\gamma,r$ sufficiently small. By assumption,
$$u(x,t) \gs u_\ast(x,t) \gs u_\ast(0,1) + a(t-1)+ \scal{p}{x} +  \frac{1}{2} \scal{Xx,x}+o(|x|^2+|t-1|).$$
With $\delta=\gamma \frac{r^2+r}{8}$, we get $u(x,t) > u_{\delta,\gamma} (x,t)$ for small $r$ and $ |x|,|t-1| \in [\frac{r}{2}, r].$ Reducing again $r$, we can assume that $u_{\delta,\gamma} < \psi^+$ on $B_r$.
Thanks to Lemma \ref{supsol}, 
$$ \tilde u(x,t)= \left\{ \begin{matrix} \max(u(x,t),u_{\delta,\gamma} (x,t)) \text{ if } |x,t-1| <r \\ u(x) \text{ otherwise} \end{matrix} \right.$$
is a subsolution of \eqref{pb} (with no initial conditions).
\end{proof}

Finally, this lemma combined with the definition of $W$ proves that $W$ is in fact a solution of \eqref{pb} (the initial conditions were already checked).

\subsection{Regularity}\label{secsha}

\begin{prop}
 The unique solution $u$ of \eqref{pb} is Lipschitz in space, with the same constant as $u_0,\psi^\pm.$
\end{prop}
\begin{proof}
We will prove that $u_{z}(x,t)=u(x+z,t) - L|z| $ is in fact a subsolution of \eqref{pb}. The Lipschitz bound is then straightforward (using the comparison principle).

To begin, we notice that $u(x+z,t)-L(|z|) \ls u^+(x,t)$ and $u(x+z,0)-L|z| \ls u_0(x+z) - L|z| \ls u_0(x).$

Assume now that $\varphi$ is any smooth function which is greater than $u_{z}$ with equality at $(\hat x, \hat t).$ Then, either, $u_{z} (\hat x,\hat t) = \psi^-(\hat x,\hat t)$ and nothing has to be done, or $u_{z} (\hat x,\hat t)> \psi^-(\hat x,\hat t)$. In the second alternative, one can write
$$u(\hat x+t,\hat t) > \psi^-(\hat x) = \psi^-(\hat x + z) + (\psi^-(\hat x) - \psi^-(\hat x +z),$$
so
\begin{equation*}u(\hat x+z,\hat t ) > \psi^-(\hat x +z) + \underbrace{\psi^-(\hat x ) - \psi^-(\hat x+z) + L|z| }_{\gs 0} \gs u^-(\hat x+z,\hat t ).\end{equation*}
As $u$ is a subsolution at $(\hat x+z,\hat t)$ and $u(x+z,t) \ls \varphi(x,t) + L|z| $ with equality at $(\hat x + z , \hat t )$, one can write with $y=x+z$, $s=t$,
$$u(y,t) \ls \varphi(y-z,s) + L|z|  := \phi(y,s),$$ with equality at $(\hat y,\hat s)$ which gives
$$\phi_t + F(D\phi(\hat x,\hat t),D^2 \phi(\hat x,\hat t)) \ls 0.$$
Since the derivatives of $\phi $ and $\varphi$ are the same, we deduce 
$$\varphi_t + F(D\varphi, D^2\varphi) \ls 0,$$
what was expected.
\end{proof}
\begin{rem}
With the same arguments, one can prove that
$$\forall \delta > 0,\quad \forall x,t,\quad | u(x,t+\delta) - u(x,t)| \ls \sup_{x} |u(x,\delta) - u(x,0)|.$$
\end{rem}

%\subsection{A link with a parabolic obstacle problem \cite{sha08}}
%

\smallskip

We now present a general regularity result 
by Shahgholian \cite{sha08}
%the link between our viscosity solutions with obstacles and a work 
which applies to viscosity solutions for parabolic equations with obstacles. 
%We show that Shahgholian's result apply to our framework. Here is the result
\begin{theo}[\cite{pet07}, Th. 4.1]
\label{thmsha}
Let $Q^+ := \{(x,t) \in \Rn \times \R \, : \, |x| < 1, t \in [0,1)\}$ and $H(u) = F(D^2 u , Du) - u_t$ where $F$ is uniformly elliptic. Let $u$ be a continuous viscosity solution of
\begin{equation}\label{sha1}
\begin{aligned}
(u - \psi) H(u)  &= 0,
\\
H(u) &\ls 0, 
\\
u &\gs \psi,
\end{aligned}\end{equation}
in $Q^+$, with boundary data
\begin{equation}
 u(x,t) = g(x,t) \gs \psi(x,t)\qquad \text{ on }\{|x|=1\} \cup \{t=0\}. \label{sha3}
\end{equation}
Assume that $\psi \in C^{1,1}(Q^+)$ and $g$ is continuous. Then, $u \in C^{1,1}$ on every compact subset of $Q^+.$
\label{theosha}
\end{theo}

It has to be noticed $H = F - \partial_t$ where $F(D^2u,Du) = -\sqrt{1+\vnau^2} \div\left( \frac{\nabla u }{\sqrt{1+\vnau^2}} \right)$ satisfies all the assumptions of \cite{sha08}, 1.3. Indeed, the uniform ellipticity is provided by the Lipschitz bound obtained in previous subsection.

Moreover, the viscosity solution $u$ of \eqref{pb} satisfies \eqref{sha1} and \eqref{sha3} on every cylinder $Q^+_r(x_0) := \{|x - x_0| \ls r, \; t \in [t_0, t_0 +r) \}$ such that $r$ is choosen sufficiently small in order to have either $Q^+_r(x_0) \cap \{u=\psi^+\} = \emptyset$ or $Q^+_r(x_0) \cap \{u=\psi^-\} = \emptyset$. In the second alternative, change every sign in the equations.

Applying Theorem \ref{theosha} we get a $C^{1,1}$ bound for $u$ on every compact subset of $Q^+_r(x_0)$. To show that $u$ is $C^{1,1}$ in the whole space, just cover $\Rn \times \R^+$ with such $Q_r^+(x_i)$.

\end{document}